\documentclass[11pt]{article}

\usepackage{epsfig,graphics}
\usepackage{amssymb,amsmath,amsthm,bm}
\usepackage{booktabs}  
\usepackage{threeparttable}  
\usepackage{multirow}

\newtheorem{theo}{Theorem}
\newtheorem{lem}{Lemma}
\newtheorem{pro}{Proposition}

\newtheorem{mrem}{Remark}
\newtheorem{mdef}{Definition}
\newtheorem{example}{Example}
\newtheorem{alg}{Algorithm}

\def\RR{{\mathbb R}}

\def\R{{\mathbb R}}

\def\fr{\mathfrak{R}}

\def\gga{\gamma}

\def\gb{\beta}

\def\gs{\sigma}
\def\gl{\lambda}
\def\wt{\widetilde}

\def\gUp{\Upsilon}

\def\ep{\epsilon}
\def\vep{\varepsilon}

\def\gt{\triangle}

\def\gp{{\prime}}
\def\b0{{\bf 0}}
\def\1{{\bf 1}}

\def\cE{\mathcal{E}}

\def\cH{\mathcal H}

\def\Err{{\rm Err}}

\def\Bd{{\rm Bd}}

\def\wh{\widehat}
\def\wt{\widetilde}

\makeatletter
\def\widebreve{\mathpalette\wide@breve}
\def\wide@breve#1#2{\sbox\z@{$#1#2$}%
     \mathop{\vbox{\m@th\ialign{##\crcr
\kern0.08em\brevefill#1{0.8\wd\z@}\crcr\noalign{\nointerlineskip}%
                    $\hss#1#2\hss$\crcr}}}\limits}
\def\brevefill#1#2{$\m@th\sbox\tw@{$#1($}%
  \hss\resizebox{#2}{\wd\tw@}{\rotatebox[origin=c]{90}{\upshape(}}\hss$}
\makeatletter
\def\wb{\widebreve}

\begin{document}

\title{Time-Scale-Chirp\_rate Operator for 
Recovery of \\ Non-stationary Signal Components with \\
Crossover Instantaneous Frequency Curves\thanks{This work was partially supported by the Hong Kong Research Council, under Projects $\sharp$ 12300917 and $\sharp$ 12303218, and HKBU Grants $\sharp$ RC-ICRS/16-17/03 and $\sharp$ RC-FNRA-IG/18-19/SCI/01, 
the Simons Foundation, under grant $\sharp$ 353185, 
and the National Natural Science Foundation of China, 
under Grants $\sharp$ 62071349, $\sharp$ 61972265  and  $\sharp$ 11871348, by National Natural Science Foundation of Guangdong Province of China,  under Grant $\sharp$ 2020B1515310008.
%, by Educational Commission of Guangdong Province of China,  under Grant $\sharp$ 2019KZDZX1007, and by Guangdong Key Laboratory of Intelligent Information Processing, China.
}
\author{Charles K. Chui${}^{1}$, 
 Qingtang Jiang${}^{2}$, Lin Li${}^{3}$, %\thanks{Corresponding author. E-mail: lilin@xidian.edu.cn.}, 
 and Jian Lu${}^{4}$
 }}
\date{}

\maketitle

\vskip -0.5cm

%\centerline
{1. Department of Mathematics, Hong Kong Baptist University, Hong Kong.   E-mail: ckchui@stanford.edu}

{2. Department of Mathematics and Statistics, Univ. of Missouri--St. Louis, St. Louis, MO. 

 \qquad  E-mail: jiangq@umsl.edu}

{3. School of Electronic Engineering, Xidian University, Xi\rq{}an, China.  E-mail: lilin@xidian.edu.cn

{4. College of Mathematics and Statistics, Shenzhen University, China. E-mail: jianlu@szu.edu.cn}

\bigskip 

%{\bf Questions:}\begin{itemize}\item[a.] Do you perfer TSC-R, TSC\_R, TSC\_r, or TSC${}_{\rm r}$ for time-scale-chirp\_rate? 
%\item[b.] Perfer: Time-Scale-Chirp\_rate signal recovery operator, Time-Scale-Chirp\_rate signal recovery operation, or Time-Scale-Chirp\_rate signal recovery transform?  
%\item[c.] The sentence in the middle on p.4 with \lq\lq{}recalling\rq\rq{} is ok? 

%\item[d.] \lq\lq{}Time-Scale-Chirp\_rate operator scheme\rq\rq{} is added on p.11. 

%\end{itemize}

\begin{abstract}
The objective of this paper is to introduce an innovative approach for the recovery of non-stationary signal components with possibly cross-over instantaneous frequency (IF) curves from a multi-component blind-source signal. The main idea is to incorporate a chirp rate parameter with the time-scale continuous wavelet-like transformation, by considering the quadratic phase representation of the signal components. Hence-forth, even if two IF curves cross, the two corresponding signal components can still be separated and recovered, provided that their chirp rates are different. In other words, signal components with the same IF value at any time instant could still be recovered. To facilitate our presentation, we introduce the notion of time-scale-chirp\_rate  (TSC-R) recovery transform or TSC-R recovery operator 
to develop a TSC-R theory for the  3-dimensional space of  time, scale, chirp rate. Our theoretical development is based on the approximation of the non-stationary signal components with linear chirps and applying the proposed adaptive TSC-R transform to the multi-component blind-source signal to obtain fairly accurate error bounds of IF estimations and signal components recovery. Several numerical experimental results are presented to demonstrate the out-performance of the proposed method over all existing time-frequency and time-scale approaches in the published literature, particularly for non-stationary source signals with crossover IFs. 
%arbitrarily close IFs. 
%The objective of this paper is to introduce an innovative approach to the recovery, from a multi-component source signal, of non-stationary signal components with possibly cross-over instantaneous frequency (IF) curves. The main idea is to incorporate a chirp-rate parameter with the time-scale transformation, by considering the quadratic phase representation of the signal components. Hence-forth, even if two IF curves cross, the two corresponding signal components can still be separated and recovered, provided that their chirp-rates are different. In other words, signal components or sub-signals with the same IF value at any time instant could till be recovered. To facilitate our presentation, we introduce the notion of time-scale, chirp-rate signal recovery (TSC-R) transform to develop a TSC-R theory, along with an effective computational scheme for the 3-dimensional  TSC-R (time-scale-chirp\_rate) space. Our theoretical development is based on the approximation of the non-stationary signal components with linear chirps and applying the proposed adaptive TSC-R transform to the multi-component source signal (or blind-source composite signal) to obtain fairly accurate error bounds of IF estimations and signal component recovery. Several numerical experimental results are presented to demonstrate the out-performance of the proposed method over all existing time-frequency and time-scale approaches in the published literature, particularly for non-stationary source signals with arbitrarily close IFs.
\end{abstract}


%\pagenumbering{arabic} \setcounter{page}{1}

\maketitle

\section{Introduction}

In nature and the current highly technological era, acquired signals are usually affected by various complicated factors and appear as multi-component (time-overlapping) modes in the form of  the adaptive harmonic model (AHM) with an additive trend function, namely:
\begin{equation}
\label{AHM0}
x(t)=A_0(t)+\sum_{k=1}^K x_k(t),  \quad x_k(t)=A_k(t) \cos \big(2\pi \phi_k(t)\big), 
\end{equation}
where $A_0(t)$ represents the trend, $A_1(t), \cdots,  A_K (t) \ge 0$ the instantaneous amplitudes (IAs), and $2\pi \phi_1(t), \cdots, 2\pi \phi_K (t)$ the instantaneous phases (IPhs), of the multi-component source signal (or composite signal) $x(t)$. The trend along with the instantaneous frequencies (IFs) $\phi^\gp_k(t)$ of $x(t)$ in  \eqref{AHM0} are often used to describe the underlying dynamics of $x(t)$. Here, the IF of the unknown component $x_k(t)$ is defined by the derivative  $\phi^\gp_k(t)$ of $1/2\pi$ multiple of the phase function. For example, radar echoes may be generated by multiple targets close to each other, or by different micro-motion parts in one target. Also, seismic signals usually consist of multiple modes that change in time with the dynamic variations of the IAs $A_k(t)$ and the  IPhs  $2\pi \phi_k(t)$, aroused by the adjacent thin layers. In many situations it is necessary to separate the multi-component source signal $x(t)$ into a finite number of mono-components $x_k(t) = A_k(t) \cos \big(2\pi \phi_k(t)\big)$ to recover the modes and underlying dynamics, implicated for the purpose of source signal processing, parameter estimation, feature extraction, and pattern recognition, etc. Unfortunately, there are very few effective rigorous methods available in the published literature for extracting or recovery of such signal components or sub-signals. 

In this regard, it is important to point out that in general, the signal decomposition approach is not suitable for resolving the inverse problem of extracting the signal components $x_1(t), \cdots, x_K(t)$ and trend $A_0(t)$ from the source data $x(t)$ in \eqref{AHM0}. In particular, although function decomposition methods in the mathematics literature are abundant, the general objective of such approach is to decompose a given function in a certain function class into its building blocks, which are not of the form of the signal components in \eqref{AHM0}. For example, in the pioneering paper \cite{RR_Coifman} by R. Coifman, the function building blocks (called atoms) do not have the phase and frequency contents as $x_k(t)=A_k(t) \cos \big(2\pi \phi_k(t)\big)$. In another pioneering paper \cite{Chen_Donoho_Saunders} by S. Chen, D. Donoho, and M. Saunders, a desired library of function building blocks is compiled to apply an innovative basis pursuit algorithm for atomic decomposition.  However, it is not feasible to compile a huge library of atoms of the form $A_k(t) \cos \big(2\pi \phi_k(t)\big)$ for arbitrary IAs and IPhs, to apply the basis pursuit algorithm for resolving the inverse problem of recovering the number $K$ of components $x_1(t), \cdots, x_K(t)$ in \eqref{AHM0}, from the source signal $x(t)$. Of course there are other well-known signal decomposition schemes, such as the discrete wavelet decomposition and sub-band coding, for signal decomposition, but they are data-independent computational schemes and definitely cannot be applied to solving this inverse problem. Even the most popular data-dependent signal (or time series) decomposition algorithm, called ``Emperical Mode Decomposition (EMD)'', proposed by N. Huang et al, as well as all variants developed by others, such as  \cite{Cicone20, HM_Zhou16, HM_Zhou20, Flandrin04, LCJJ19, HM_Zhou09, Rilling08, Walt, Walt2, %van20,
Y_Wang12, Wu_Huang09, Xu06, Li_Ji09}, fail in resolving this inverse problem. The reason is that there is absolutely no reason for the EMD decomposed components, called intrinsic mode functions (IMFs), to possess any phase and frequency information. After all, the manipulation of applying the Hilbert transform to analytically extend each  IMF from the real line to the upper half plane, followed by taking the real part of the polar formulation of the extension to obtain the instantaneous phase representation, can also be applied to any arbitrary integrable function. In fact, the derivative of this artificial instantaneous phase function is not necessarily positive (for the formulation of the instantaneous frequency), even if the derivative exists. 

On the other hand, it would be much more reasonable to first extract the (instantaneous) frequencies, and then using the frequency contents to “decompose” the source signal into its components. Let us call this procedure the ``signal resolution''  approach. In other words, the signal resolution approach is a logical way to resolve the inverse problem using the data information from the source signal. For stationary signals (that is, source signals with linear-phase components), the signal resolution approach has a very long history, dated back to B.G.R. De Prony, who introduced the Prony method in his 1795 paper \cite{De Prony} to solving the inverse problem of time-invariant linear systems with constant coefficients. This pioneering paper stimulates the development of two very important and popular algorithms, called ``MUSIC\rq\rq{} proposed by R.O. Schmidt  in \cite{MUSIC} and ``ESPRIT\rq\rq{} introduced by R. Roy and Kailath in \cite{ESPRIT}. While the number $K$ of signal components of the stationary model \eqref{AHM0} with linear phases and constant coefficients is needed for carrying out the Prony method, it is not necessary for both MUSIC and ESPRIT, even with non-constant coefficients in \eqref{AHM0}.  

The first signal resolution approach for non-stationary signals, coined ``synchrosqueezed transform (SST)\rq\rq{}  by I. Daubechies and S. Maes in \cite{Daub_Maes96} and studied by H.-T. Wu in his Ph.D. dissertation \cite{Wu_thesis}, where both the continuous wavelet transform (CWT) and the short-time Fourier transform (STFT) are considered to compute some reference frequency from the source signal for the SST operation to squeeze out the instantaneous frequencies (IFs) of the signal components. The full developments of SST using CWT and STFT are published in \cite{Daub_Lu_Wu11} and \cite{Thakur_Wu11}, respectively. One of the limitations of the SST approach is the need of sufficiently accurate IFs for applying the normalized integral of the SST output in a small neighborhood of each IF to recover the signal components, but without assurance of the number $K$ of such IFs or signal components in \eqref{AHM0}. Further development and study in the area of SST and its applications include the more recent publications \cite{ Flandrin_Wu_etal_review13,Chui_Lin_Wu15, Chui_Walt15, Daub_Wang_Wu15, Yang15, Yang14, MOM14, OM17, BMO18, Wu17, Saito17, LCHJJ18, LCJ18, CJLS18, LJL20, Li_Liang12, Wang_etal14,Jiang_Suter17, WCSGTZ18}. 
%MOM15,OM17, BMO18, Pham17, Lin_Cubic19}. 
More recently, another time-frequency approach, coined ``signal separation operation or operator (SSO)\rq\rq{} by the first author and H.N. Mhaskar in the joint work \cite{Chui_Mhaskar15}  for resolving the inverse problem \eqref{AHM0} by using discrete data acquired from the source multi-component signal. In contrast to SST, the SSO is a direct method for recovering the signal components simply by plugging the computed IF values in the same SSO (operator). Further development in the direction of SSO includes \cite{Chui_Mhaskar_Walt, LCJ20, CJLL20_adpSTFT, Chui_Han_Mhaskar, Chui_Mhaskar}. In the literature, both SST and SSO are commonly called ``time- frequency\rq\rq{} approaches. Another consideration of the signal resolution approach is the ``time-scale\rq\rq{} approach by using the CWT and recalling that the scale parameter of the CWT is inversely proportional to the frequency to be estimated by the CWT. Of course the constants of inverse proportionality depend on the choice of the analysis wavelets for the CWT. In very recent paper \cite{Chui_Han20}, the classical Haar function is extended to a family of cardinal splines, called extended Haar wavelets $\psi_{m, n}(x)$, with any desirable polynomial spline order $m\ge1$, any desirable order $n\ge1$ of vanishing moments, and compact supports $[-(m+n)2, (m+n)/2]$, for which the constants of inverse proportionality are easily computed (see Equation (2.2) and Tables 3--5 in \cite{Chui_Han20}). One advantage of the time-scale approach proposed in \cite{Chui_Han20} over the SSO is the elimination of the additional parameter for estimating the IFs of the signal components.  

Observe that in applying %SST,  
SSO and the time-scale approach in \cite{Chui_Han20},  the phase functions of the signal components in the source signal model \eqref{AHM0} are approximated by some linear polynomials at any local time for the purpose of extracting the IFs. More recently, 
%quadratic and high-order approximations have also been studied in \cite{MOM15,OM17, BMO18, Pham17, Lin_Cubic19}. In particular, the consideration of 
quadratic approximation at local time gives rise to the SSO of ``linear chirp-based model\rq{}\rq{} proposed in our paper  \cite{LCJ20}. This model provides a more accurate component recovery formulae, with theoretical analysis established in our recent work \cite{CJLL20_adpSTFT}.  
The main reason for considering the quadratic terms of the phase approximation is to recover signal components with the same IF values. In this regard, we emphasize that in the current literature, including all time-frequency and time-scale approaches, the IFs of the signal components are assumed to be distinct and well separated. This strict assumption must be removed in order to apply the methods and algorithms to separate more general real-world multi-component or composite signals. To demonstrate this point of view, let us consider radar signal processing, where the micro-Doppler effects are represented by highly non-stationary signals. When the target or any structure on the target undergoes micro-motion dynamics, such as mechanical vibrations, rotations, or tumbling and coning motions \cite{radar_basic_2016, Stankovic_compressive_sense_2013}, the frequency curves of the signal components may cross one another. For example, Fig.\ref{figure:Micro_Doppler} shows the simulated micro-Doppler modulations (that is, two sinusoidal frequency-modulation signals and one single-tone signal) and the STFT of the synthetic signal. 

%%%%%%%%%%%%%%%%%%%the beginning of figure 1 %%%%%%%%%%%%%%%
\begin{figure}[th]
	\centering
	%\begin{tabular}{ccc}
	\begin{tabular}{cc}
		\resizebox {2.4in}{1.8in} {\includegraphics{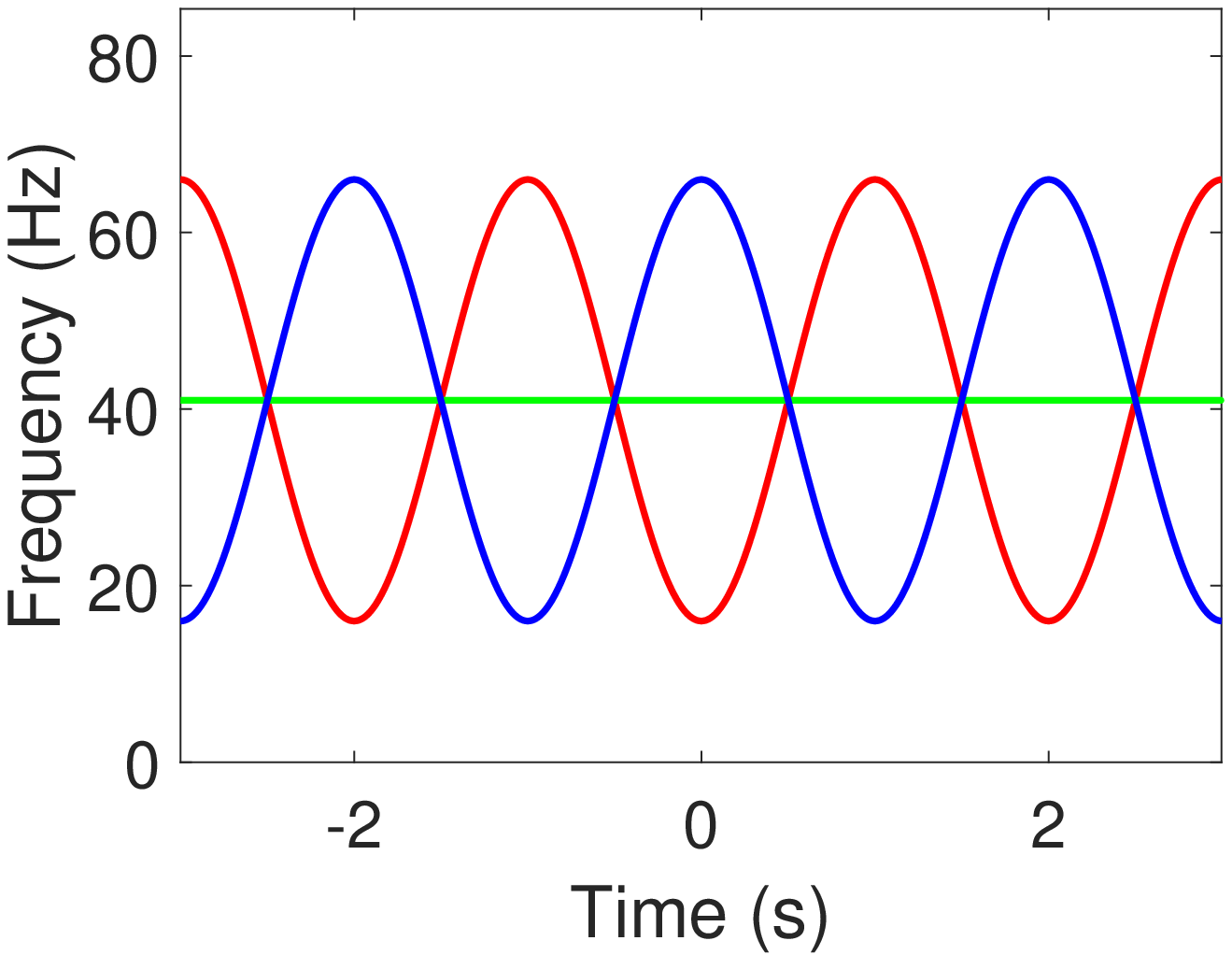}}
		\quad & \quad 
		\resizebox {2.4in}{1.8in} {\includegraphics{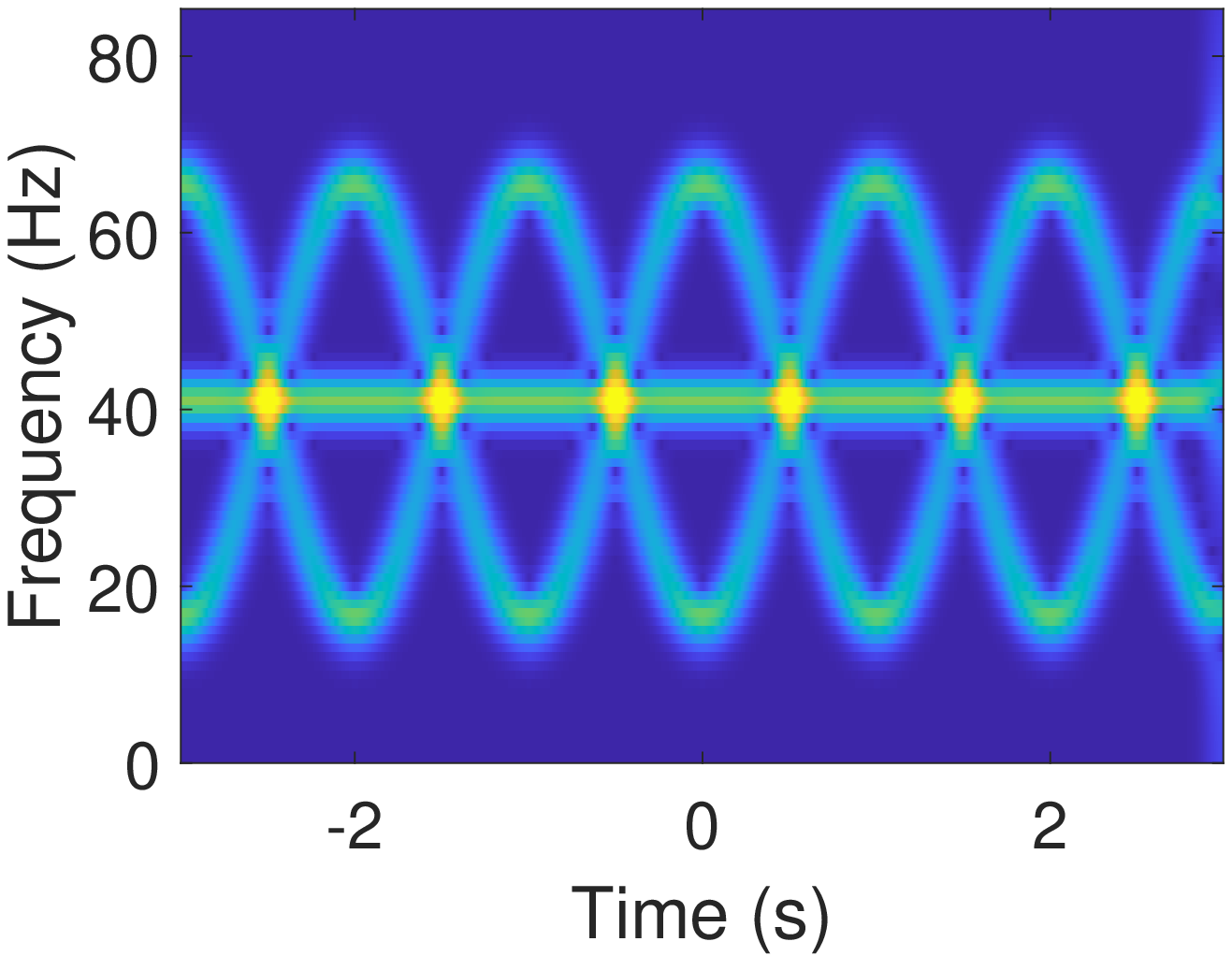}}
	\end{tabular}
	\vskip -0.3cm 
	\caption{\small  Micro-Doppler modulations induced by target\rq{}s tumbling (Left) and  STFT of the signal (Right).}
	%their STFTs (Left).} %from \cite{radar_basic_2016}}
	\label{figure:Micro_Doppler}
\end{figure}
%%%%%%%%%%%%%%%%the end of figure 1  %%%%%%%%%%%%%%%%%%%%%

%More precisely, two signal components $x_k(t)$ and $x_\ell(t)$ of a multi-component signal $x(t)$ governed by \eqref{AHM0} are said to overlap in the time-frequency plane at $t=t_0$ if the two IF curves cross at some $t=t_0$; that is, 
%$$\phi_k'(t_0)=\phi_\ell'(t_0). % \; \hbox{and $\phi_k''(t)\not = \phi_\ell''(t)$ for $t\in [t_0-\gd_1, t_0+\gd_1]$}, 
%$$

To be precise, we say that two signal components $x_k(t)$ and $x_\ell(t)$ of a multi-component signal $x(t)$ governed by \eqref{AHM0} overlap in the time-frequency plane at $t=t_0$, if $\phi_k^\gp(t_0) = \phi_\ell^\gp(t_0)$ but  $\phi_k^\gp(t)\not = \phi_k^\gp(t)$ in some deleted neighborhood of $t_0$.
Based on the linear chirp-based model proposed in our previous paper  \cite{LCJ20}, we have extended the SSO method in \cite{Chui_Mhaskar15} by incorporating a chirp rate parameter to introduce a computational scheme in our work \cite{LHJC20} for the recovery of signal components with overlapping frequency curves. 
In the present paper,  we propose another innovative time-scale approach by introducing a 3D time-scale-chirp\_rate transform, formulated by incorporating a complex quadratic phase function with a continuous wavelet-like transform (CWLT), 
%we propose a 3-dimensional time-scale-chirp\_rate operation by introducing an innovative continuous wavelet-like transform (CWLT) with a complex quadratic phase function, 
to be called an adaptive ``time-scale-chirp\_rate (TSC-R)\rq{}\rq{} component recovery operator, and develop a rigorous theory for assurance of solving the inverse problem in separating the signal components $x_k(t)$ of the multi-component signal $x(t)$ governed by \eqref{AHM0}, without the assumption of well separated IFs, but rather by assuming that 
if the two IF curves of the signal components $x_k$ and $x_{\ell}$ cross at some $t=t_0$, then $|\phi''_{k}(t)-\phi''_{\ell}(t)| \ge \delta$ for some $\delta > 0$, for $|t - t_0| <\epsilon$, where $\epsilon > 0$.

For convenience, we will consider, without loss of generality, the following complex-version of \eqref{AHM0} without the trend function $A_0(t)$ function, namely:
\begin{equation}
  \label{AHM}
  x(t)=\sum_{k=1}^K x_k(t)=\sum_{k=1}^K A_k(t) e^{i2\pi\phi_k(t)}
  \end{equation}
where $A_k(t), \phi_k'(t)>0$.  The reader is referred to \cite{Chui_Mhaskar15} for the methods of polynomial trend removal. 

The presentation of this paper is organized as follows. In Section 2, the adaptive TSC-R operator is introduced and developed, along with some error bounds, for instantaneous frequency estimation and signal components recovery. When the Gaussian function is used as the wavelet-like scalable window, more precise error bounds are derived for the adaptive TSC-R operation in Section 3. Numerical experimental results will be discussed in Section 4.

\section {Time-scale-chirp\_rate signal recovery operator}

To extract and separate the (unknown) signal components with crossover IFs from the multi-component signal governed by \eqref{AHM}, we propose the following adaptive {\bf time-scale-chirp\_rate  signal recovery} (TSC-R) operator, by introducing an adaptive  continuous wavelet-like transform (CWLT), namely: 
	\begin{eqnarray}
	\nonumber
	U_x(a, b, \gl) \hskip -0.6cm  && := \int_{-\infty}^\infty x(t) \frac 1{a}g\big(\frac{t-b}{a\gs(b)}\big)
	e^{-i2\pi \mu  \frac{t-b}{a}}
	  e^{ -i\pi \lambda (t-b)^2} dt\\
\label{def_adaptiveTSC-R}	&&= \int_{\RR} x(b+at)  \frac 1 {\gs(b)} g\big(\frac t {\gs(b)}\big) e^{-i2\pi \mu  t -i\pi \gl a^2 t^2}dt,
	\end{eqnarray}
where $g(t)$ is a window function, $\mu$ is a positive constant, and $\gs(b)$ is a positive function of $b$.  In this paper, all window functions $g$  are assumed to be functions in $L_2(\R)$ that decay to zero at $\infty$ and satisfy $\int_\RR g(t) dt =1$.  Observe that when $\lambda=0$,  $U_x(a, b, \gl)$ is reduced to the adaptive CWLT of $x(t)$, 
denoted by $\wt W_x(a, b)$, as considered in \cite{LCJ18}, and that the TSC-R of $x(t)$ can be considered as a multi-component signal in the 3-dimensional space of time $t$, scale $a$, and chirp rate $\gl$. The importance of this transform is that when the IF curves of two components $x_k(t)$ and $x_\ell(t)$ cross each other, they may be well-separated in the 3-dimensional space by adaptive TSC-R operator, provided that  $\phi_k''(t)\not = \phi_\ell''(t)$ for $t$ in some neighborhood of the cross-over time instant $t_0$.  Thus, a multi-component signal $x(t)$ with certain signal components that have the same IF values can be extracted and well-separated in the 3-dimensional TSC-R space adaptively. Hence, it is feasible to reconstruct signal components by adaptive TSC-R.

In practice, for a particular signal $x(t)$, its %CWLT $W_x(a,b)$ 
%or 
adaptive CWLT $\wt W_x(a, b)$ 
lies in a region of the scale-time plane: 
$$
\{(a, b): \; a_1(b)\le a\le a_2(b), b\in \RR\}
$$ 
for some $0<a_1(b),  a_2(b) <\infty$. That is $\wt W_x(a,b)$ 
is negligible for  $(a, b)$ outside this region. Throughout this paper we assume for each $b\in \RR$, the scale $a$ is in the interval:
\begin{equation}
\label{a_interval} 
a_1(b)\le a\le a_2(b). 
\end{equation}	

\begin{mdef} 
	\label{def:function_class}
	For $\ep_1>0$ and $\ep_3>0$, let $\cE_{\ep_1,\ep_3}$ denote the set consisting of (complex) adaptive harmonic models (AHMs) defined by \eqref{AHM} with $A_k(t) \in L_\infty(\RR), \; A_k(t)>0, \phi_k(t)\in C^3(\RR), \inf_{b\in \RR} \phi_k'(t)>0, \sup_{b\in \RR} \phi_k'(t)<\infty$,   and 
 $A_k(t), \phi_k(t)$ satisfying %\eqref{def_sep_cond_cros} for some $\rho\ge 0$ and $\gt>0$, 
\begin{eqnarray}
\label{cond_A}
&&	|A_k(t+\tau)-A_k(t)|\leq \vep_1 |\tau|A_{k}(t),~~t\in \RR, \; k=1,\cdots, K, \\
\label{cond_phi}
&&	|\phi'''_{k}(t)| \leq \vep_3,  ~~t\in \RR, \; k=1, \cdots, K. 
	\end{eqnarray}
\end{mdef}

\bigskip 
	
For  a window function $g\in L_1(\RR)$, denote 
\begin{equation}
\label{def_PFT}
\wb g (\eta, \gl):=\int_{\RR} g(t) e^{-i2\pi\eta t-i\pi \gl t^2}dt. 
\end{equation}
$\wb g(\eta, \gl)$ is called a polynomial Fourier transform of $g$ \cite{Bi_Stankovic11,Stankovic13}. Note that 
$\wb g (0, 0)=1$ since $\int_\R g(t) dt=1$. 

When $g$ is the Gaussian function defined by 
\begin{equation}
\label{def_g}
g(t)=\frac 1{\sqrt {2\pi}} \; e^{-\frac {t^2}2},   
\end{equation}
then we have (refer to \cite{Leon_Cohen, LCHJJ18}) 
%\cite{Leon_Cohen,Gibson06,LCHJJ18, LCJ18})
\begin{equation}
\label{g_PFT} 
\wb g(\eta, \gl)=\frac 1{\sqrt{1+i2\pi\gl}} e^{-\frac{2\pi^2 \eta ^2}{1+i2\pi \gl}},   
\end{equation}
where $\sqrt{1+i2\pi\gl}$ denotes the square root of $1+i2\pi\gl$ lying in the same quadrant as $1+i2\pi\gl$.

\bigskip 

We say $s(t)$ is a linear chirp or a linear frequency modulation signal if
\begin{equation*}
%\label{def_chirp}
s(t)=A e^{i2\pi \phi(t)}=A e^{i2\pi (ct +\frac 12 r t^2)},   
\end{equation*}
where $c$ and $r$ are constants. We use linear chirps to approximate each $x_k(t)$ at any local time. Namely, we write 
$$
x_k(b+a t)=x_k(b)e^{i2\pi (\phi'_k(b) a t +\frac 12\phi''_k(b) (at)^2) }+ x_{{\rm r}, k}(a, b, t),
$$
where 
$$
 x_{{\rm r}, k}(a, b, t)= x_k(b+a t)-x_k(b)e^{i2\pi (\phi'_k(b) a t +\frac 12\phi''_k(b) (at)^2) }.
$$
Note that, as a function of $t$, $x_k(b)e^{i2\pi (\phi'_k(b)at+\frac 12\phi''_k(b)a^2t^2) }$ is a linear chirp. 
Thus $x(b+at)$ can be approximated by a superposition of linear chirps at any local time $t$:
\begin{equation*}
%\label{xm_xr}
x(b+a t)=x_{\rm m}(a,b,t)+x_{\rm r}(a,b,t),
\end{equation*}
where
\begin{eqnarray*}
%\label{xm}
&&x_{\rm m}(a,b,t):=\sum_{k=1}^K x_k(b)e^{i2\pi (\phi_k'(b) at+\frac 12\phi''_k(b) (at)^2) }\; , \\
%\label{xr}
&&x_{\rm r}(a,b,t):=\sum_{k=1}^K  x_{{\rm r}, k}(a, b, t). 
\end{eqnarray*}

Denote  
\begin{equation}
	\label{def_MSSO_m1}
	 \fr_x(a, b, \gl) :=  
 \int_{\RR} x_{\rm m}(a, b, t)  \frac 1 {\gs(b)} g\big(\frac t {\gs(b)}\big)  e^{-i2\pi \mu t-i\pi \lambda a^2 t^2} dt. 
	\end{equation}
Then we have 
\begin{equation}
\label{def_MSSO_m2}
 \fr_x(a, b, \gl) =\sum_{k=1}^K x_k(b) \wb g\Big( \gs(b)\big(\mu -a \phi'_k(b)\big), \gs^2(b)a^2\big (\gl - \phi''_k(b)\big)\Big). 
	\end{equation}

In the following, we denote 
 \begin{equation}\label{def_M_u}
\nu=\nu(b):= \min_{1\leq k\leq K} A_{k}(b), ~~ M=M(b):=\sum_{k=1}^K A_{k}(b).  
\end{equation}
%Throughout this paper, $\sum_{k\not= \ell}$ denotes $\sum_{\{k: ~ k\not= \ell, 1\le k\le K\}}$.

In the next lemma we provide an error bound for $|U_x(a, b, \gl)-\fr_x(a, b, \gl)|$. 

\begin{lem}\label{lem1} 
Let $x(t)\in \cE_{\ep_1,\ep_3}$ for some $\ep_1>0, \ep_3>0$, and let $U_x(a, b, \gl)$ be its adaptive TSC-R defined by \eqref{def_adaptiveTSC-R} with a window function $g$ and $\fr_x(a, b, \gl)$ the approximation of $U_x(a, b, \gl)$ defined by \eqref{def_MSSO_m1}. Then 
\begin{equation}
\label{S_R_error}
\big|U_x(a, b, \gl)-\fr_x(a, b, \gl)\big|\le M(b)\Pi(a, b), 
\end{equation}   
where 
\begin{equation}
\label{def_Pi0}
\Pi(a, b):=\vep_1 I_1 a \gs(b) + \frac {\pi}3\vep_3  I_3 a^3 \gs^3(b), 
\end{equation} 
 {\rm with}
   \begin{equation}
\label{def_In}
 I_n:=\int_\RR \big|g(t) t^n\big| dt, \; n=1, 2, \cdots. 
\end{equation}
\end{lem}
\begin{proof}  By \eqref{cond_A} and \eqref{cond_phi}, 
\begin{eqnarray*}
&&|x(b+at)-x_{\rm m}(a, b, t)|=|x_{\rm r}(a, b, t)|\\
&&=\sum_{k=1}^K \Big\{(A_k(b+at)-A_k(b))e^{i2\pi \phi_k(b+at)}
\\&&\hskip 1cm 
+x_k(b)e^{i2\pi (\phi'_k(b) at+\frac 12\phi''_k(b)(at)^2) }
\big(
e^{i2\pi (\phi_k(b+at)-\phi_k(b)-\phi_k'(b) at - \frac 12\phi''_k(b)(at)^2)}-1\big)\Big\}\\
&& \le \sum_{k=1}^K\Big\{ 
\Big|A_k(b+at)-A_k(b)\Big| +A_k(b) \Big |i2\pi \Big(\phi_k(b+at)-\phi_k(b)-\phi_k'(b) at - \frac 12\phi''_k(b)(at)^2\Big) 
\Big|\Big\}\\
&& \le \sum_{k=1}^K\Big\{ A_k(b) \vep_1 a |t|  + A_k(b)  2\pi \sup_{\xi \in \RR}  \frac 16 \big |\phi'''_k(\xi) (a t)^3 \big|\Big\}\\
&&\le  M(b) \vep_1 a |t|  + M(b) \frac \pi 3 \vep_3  a^3 |t|^3. 
\end{eqnarray*}
This leads to 
\begin{eqnarray*}
&&\big|U_x(a, b, \gl)-\fr_x(a, b, \gl)\big| =
\Big| \int_\RR (x(b+at)-x_{\rm m}(a, b, t))\frac 1{\gs(b)}g(\frac t{\gs(b)}) e^{-i2\pi \mu  a t-i\pi \gl (a t)^2}dt \Big|\\
&&\qquad  \le\int_\RR  M(b) \big(\vep_1 a|t|  + \frac \pi 3 \vep_3  a^3|t|^3\big)
\big|\frac 1{\gs(b)}g(\frac t{\gs(b)}) \big|dt \\
&&\qquad =M(b)\big(\vep_1 I_1 \gs(b) + \frac {\pi}3\vep_3  I_3 \gs^3(b)\big). 
\end{eqnarray*}
This completes the proof of  Lemma \ref{lem1}.
\end{proof}

\begin{mrem} 
Recall that we assume in this paper $a$ is in an interval as shown in \eqref{a_interval}. Thus, we have 
 \begin{equation}
\label{S_R_error_interal}
\big|U_x(a, b, \gl)-\fr_x(a, b, \gl)\big|\le M(b)\Pi_0(b), 
\end{equation} 
where 
\begin{equation}
\label{def_Pi_large}
\Pi_0(b):=\Pi(a_2,b)=\vep_1 I_1 \gs(b)a_2
+ \frac {\pi}3\vep_3  I_3 \gs^3(b)a_2^3.  
\end{equation} 
\hfill $\blacksquare$ 
\end{mrem}

In the following, we assume  any two IF curves of the signal components $x_k$ and $x_{\ell}$ satsify %cross at some $t=t_0$,
\begin{equation}\label{def_sep_cond_cros}
 \hbox{either} ~~	\frac {|\phi'_{k}(t)-\phi'_{\ell}(t)|}{\phi'_{k}(t)+\phi'_{\ell}(t)}\ge \gt,  \; t\in \RR,  ~~  \hbox{or} ~~ |\phi''_{k}(t)-\phi''_{\ell}(t)| \ge 2 \gt_1, \; t\in \RR,  
	\end{equation}
where $0<\gt<1, \gt_1>0$.  Clearly $\phi'_{k}(t)$ and $\phi'_{\ell}(t)$ could be cross over at a time instant. 
For $1\le k\le K$, define  
\begin{equation}
\label{def_Zk}
Z_k:=\{(a, b, \gl):  \; 
	|\mu-a \phi'_k(b)|<\gt ~ \hbox{and} ~ |\gl-\phi''_k(b)| <  \gt_1, \; b\in \RR\}.     
\end{equation}  
%where $\gt$ and $\gt_1$ are the frequency resolutions in \eqref{def_sep_cond_cros} with $0<\gt<1, \gt_1>0$. 

\begin{lem}\label{lem:Zk_disjoint}
 If $\phi_k$ satisfies \eqref{def_sep_cond_cros}, then $Z_k, 1\le k\le K$ are disjoint, that is $Z_\ell \cap Z_{ k}=\emptyset$ for $\ell\not=k$. 
\end{lem} 
The proof of Lemma \ref{lem:Zk_disjoint} is straightforward and it is omitted. 

\bigskip 

Note that for $(a, b, \gl) \in Z_\ell$, the scale variable $a$ satisfies 
$$
\frac {\mu-\gt}{\phi'_\ell(b)}< a <\frac {\mu+\gt}{\phi'_\ell(b)}. 
$$
Hence  for any $(a, b, \gl) \in Z_\ell$,  by \eqref{S_R_error}, we have 
\begin{equation}
\label{S_R_error1}
\big|U_x(a, b, \gl)-\fr_x(a, b, \gl)\big|\le M(b)\Pi_\ell(b), 
\end{equation}   
where 
\begin{equation}
\label{def_Pi_ell}
\Pi_\ell(b):=\Pi(\frac {\mu+\gt}{\phi'_\ell(b)},b)=\vep_1 I_1 \gs(b)\frac {\mu+\gt}{\phi'_\ell(b)} 
+ \frac {\pi}3\vep_3  I_3 \gs^3(b)\Big(\frac {\mu+\gt}{\phi'_\ell(b)}\Big)^3. 
\end{equation}

\bigskip 

For a fixed $b$, and a positive number  $\wt \ep_1$, 
we let $\cH_b$ and $\cH_{b, k}$ denote the sets defined by 
\begin{equation}
\label{def_cGk}
\begin{array}{l}
\cH_b:=\big\{(a, \gl): \; |U_x(a, b, \gl)|>\wt \ep_1\big\}, \\
\cH_{b, k}:=\Big\{(a, \gl) \in \cH_b: \;  
|\mu-a \phi'_k(b)|<\gt ~ \hbox{and} ~ |\gl-\phi''_k(b)| <  \gt_1\Big\}.  
\end{array}
\end{equation}
Note that $\cH_b$ and $\cH_{b, k}$ depend on $\wt \ep_1$, and for simplicity of presentation, we drop  
$\wt \ep_1$ from them. 

Let $\Upsilon(b), \Upsilon_{\ell, k}(b)$ with $\Upsilon(b)\ge \Upsilon_{\ell, k}(b)$ for $k\not=\ell$  
be some functions satisfying 
 \begin{equation}
 \label{def_upper_bounds}
 \begin{array}{l}
 \sup_{\{(a, \gl): (a, b, \gl)\not \in \cup_{k=1}^K Z_k\}}\big|\wb g\big( \gs(b)(\mu -a \phi'_k(b)), \gs^2(b)a^2 (\gl - \phi''_k(b))\big)\big|\le \gUp(b), \\
\sup_{\{(a, \gl): (a, b, \gl)\not \in Z_\ell\}}
\big|\wb g\big( \gs(b)(\mu -a \phi'_k(b)), \gs^2(b)a^2 (\gl - \phi''_k(b))\big)\big | \le \gUp_{\ell,  k}(b). 
 \end{array}
 \end{equation}
 About the quantities $\Upsilon(b)$ and $\Upsilon_{\ell, k}(b)$, refer to Section 3 when $g$ is the Gaussian window function. 
 
\bigskip 
Next we provide another lemma which will be used to derive our main theorem. In the following lemma and the rest of this paper, $\sum_{k\not= \ell}$ denotes $\sum_{\{k: ~ k\not= \ell, 1\le k\le K\}}$.

\begin{lem}
Let $x(t)\in \cE_{\ep_1,\ep_3}$ for some $\ep_1>0,\ep_3>0$, and $U_x(a, b, \gl)$ be the adaptive TSC-R of $x(t)$ with a  window function $g$.  Then for any  $(a, \gl)\in \cH_{b, \ell}$,
\begin{eqnarray}
\label{U_with_ell}
&&\big|U_x(a, b, \gl)-x_\ell(b) \wb g\big( \gs(b)(\mu -a \phi'_\ell(b)), \gs^2(b)a^2(\gl - \phi''_\ell(b))\big) \big|\le \Err_\ell(b),   
\end{eqnarray}   
where 
\begin{equation}
\label{def_Err}
\Err_\ell(b):=M(b)\Pi_\ell(b)+\sum_{k\not= \ell} A_k(b) \gUp_{\ell, k}(b)
\end{equation}
with $\Pi_\ell(b)$ defined by \eqref{def_Pi_ell}. 
\end{lem}
\begin{proof} By \eqref{def_MSSO_m2}, we have for any $(a, \gl)\in \cH_{b, \ell}$, 
\begin{eqnarray*}
&&	\big|\fr_x(a, b, \gl)-x_\ell(b) \wb g\big( \gs(b)(\mu -a \phi'_\ell(b)), \gs^2(b)a^2(\gl - \phi''_\ell(b))\big) \big| \\
&& =\Big|\sum_{k\not= \ell } x_k(b) \wb g\big( \gs(b)(\mu -a \phi'_k(b)), \gs^2(b)a^2 (\gl - \phi''_k(b))\big)\Big|\\
&& \le \sum_{k\not= \ell } A_k(b)  \Big| \wb g\big( \gs(b)(\mu -a \phi'_k(b)), \gs^2(b)a^2 (\gl - \phi''_k(b))\big)\Big| 
\le \sum_{k\not= \ell } A_k(b)  \gUp_{\ell, k}(b).  
\end{eqnarray*}
This, along with \eqref{S_R_error1}, leads to that 
\begin{eqnarray*}
&&	\hbox{Left hand side of \eqref{U_with_ell}}\\
&& \le   \big|U_x(a, b, \gl)-\fr_x(a, b, \gl)\big|+ 
\big|\fr_x(a, b, \gl)-x_\ell(b) \wb g\big( \gs(b)(\mu -a \phi'_\ell(b)), \gs^2(b)a^2(\gl - \phi''_\ell(b))\big) \big| \\
&& \le M(b)\Pi_\ell(b)+\sum_{k\not= \ell } A_k(b)  \gUp_{\ell, k}(b).  
\end{eqnarray*}
Thus \eqref{U_with_ell} holds true. 
\end{proof}

Next we have the following theorem. 
\begin{theo}\label{theo:adaptive TSC-R1}
Let $x(t)\in \cE_{\ep_1,\ep_3}$ for some $\ep_1>0,\ep_3>0$, and $U_x(a, b, \gl)$ be the adaptive TSC-R of $x(t)$ with a  window function $g$.  Suppose $x(t)$ satisfies \eqref{def_sep_cond_cros} for some $0<\gt<1$ and $\gt_1>0$, and 
\begin{equation}
\label{theo1_cond1}
2M(b)\big(\gUp(b)+\Pi_0(b)\big)\le \nu(b) 
\end{equation}
holds.  Let $\cH_b$ and $\cH_{b, k}$ be the sets defined by \eqref{def_cGk} with a function $\wt \ep_1=\wt \ep_1(b)>0$ satisfying   
\begin{equation}
\label{cond_ep1}
M(b)\big(\gUp(b)+\Pi_0(b)\big)\le 
\wt \ep_1 \le \nu(b)-M(b)\big(\gUp(b)+\Pi_0(b)\big).
\end{equation}
Then the following statements hold.
\begin{enumerate}
\item[{\rm (a)}] $\cH_b=\cup_{k=1}^K \cH_{b, k}$.  
\item[{\rm (b)}]  The sets $\cH_{b, k}, 1\le k\le K$ are disjoint, i.e. $\cH_{b, k}\cap {\cH}_{b, k\rq{}}=\emptyset$ if $k\not= k\rq{}$. 
\item[{\rm (c)}] Each set $\cH_{b, k}$ is non-empty. 
\end{enumerate}
\end{theo}

We delay the proof of Theorem \ref{theo:adaptive TSC-R1} to the end of this section.

\bigskip 
Denote 
\begin{equation}
\label{def_max_eta}
 (\wh a_\ell, \wh \gl_\ell) =(\wh a_\ell(b), \wh \gl_\ell(b)):={\rm argmax}_{(a, \gl) \in\cH_{b, \ell}  }|U_x(a, b,  \gl)|, ~~ \ell=1, \cdots,  K.
\end{equation}
From Theorem \ref{theo:adaptive TSC-R1},  we know $\wh a_\ell(b)$ and $\wh \gl_\ell(b)$ are well defined. We will use them to estimate $\phi'_\ell(b)$, chirp rate $\phi^{\gp\gp}_\ell(b)$ and to recover $x_\ell(b)$.  More precisely, we have the following 
TSC\_R operator scheme for IF estimation and component recovery. 

\begin{alg} {\bf (Time-Scale-Chirp\_rate operator scheme)} \; Suppose  $x(t)\in \cE_{\ep_1,\ep_3}$ satisfies the conditions in Theorem \ref{theo:adaptive TSC-R1}. 
\begin{itemize}
\item[] {\bf Step 1.} Calculate  $\wh a_\ell(b)$ and $\wh \gl_\ell(b)$ by \eqref{def_max_eta}.
\item[] {\bf Step 2.} Obtain IF and chirp rate estimates by 
\begin{equation}
\label{IF_estimate}
\phi'_\ell(b) \approx \frac \mu{\wh a_\ell(b)}, \quad 
\phi^{\gp\gp}_\ell(b) \approx \wh \gl_\ell(b), 
\end{equation} 
\item[] {\bf Step 3.} Obtain the recovered $\ell$-th component by 
\begin{equation}
\label{comp_recover}
x_\ell(b)\approx U_x(\wh a_\ell,  b, \wh \gl_\ell).  
\end{equation}
\hfill $\blacksquare$ 
\end{itemize}
\end{alg}

Observe that the recovered component is obtained  
simply  by substituting the time-scale ridge $\wh a_\ell(b)$ and time-chirp rate ridge $\wh \gl_\ell(b)$ to adaptive TSC-R, which is different from SST method with which the recovered $x_k(t)$ is computed by a definite integral along each estimated IF curve on the SST plane.  

Next we study the error bounds for these approximations. To this regard, we introduce admissible window functions.  

\begin{mdef}\label{definition2} {\rm ({\bf Admissible window function})} \; 
A function  $g(t)$ in $L_2(\R)$  
is called an admissible window function if $\int_\RR g(t)dt=1$,  $g$ has certain at $\infty$ and satisfies the following conditions.
\begin{itemize}
\item[{\rm (a)}]  $|\wb g(\eta, \gl)|$ can be written as $f(|\eta|, |\gl|)$ for some function $f(\xi_1, \xi_2)$ defined on $0\le \xi_1, \xi_2< \infty$. 

\item[{\rm (b)}] There exist $c_0$ with $0<c_0<1$ and (strictly) decreasing non-negative continuous functions $\gb(\xi)$ and $\gga(\xi)$ on $[0, \infty)$ with $\gb(0)=1$, $\gga(0)=1$ 
such that if $f$ in {\rm (a)} satisfies 
\begin{equation}
\label{ineq_cond}
1-c\le f(\eta, \gl), 
\end{equation} 
for some $c$ with $0\le c \le c_0$ and $\eta, \gl$, then 
 \begin{equation}
\label{cond_gb_gga}
 1-c\le \gb(\eta), \quad  1-c\le \gga(\gl).
\end{equation}  
\end{itemize}
\end{mdef}

\begin{theo}\label{theo:adaptive TSC-R2}
Let $x(t)\in \cE_{\ep_1,\ep_3}$ for some $\ep_1>0,\ep_3>0$, and $U_x(a, b, \gl)$ be the adaptive TSC-R of $x(t)$ with an admissible window function $g$ for certain $c_0$ such that \eqref{cond_gb_gga} holds.  Suppose \eqref{def_sep_cond_cros} 
and \eqref{theo1_cond1} hold and that for $1\le \ell \le K$, $2\Err_\ell(b)/A_\ell(b)\le c_0$, where $\Err_\ell(b)$ is defined by \eqref{def_Err}. Let $\cH_b$ and $\cH_{b, k}$ be the sets defined by \eqref{def_cGk} for some 
$\wt \ep_1$ satisfying  \eqref{cond_ep1}. Let $\wh a_\ell(b), \wh \gl_\ell(b)$ be the functions defined by \eqref{def_max_eta}. Then the following statements hold.
\begin{enumerate}
\item[{\rm (a)}] For $\ell=1, 2, \cdots, K$, 
\begin{eqnarray}
&&\label{phi_est}
|\mu - \wh a_{\ell}(b)\phi_{\ell}'(b)|\le \Bd_{1, \ell}:=\frac{1}{\gs(b)} \gb^{-1}\big(1-\frac {2 \; \Err_\ell(b)}{A_\ell(b)}\big), \\
&&\label{phi_est_gl}
|\wh\gl_{\ell}(b)-\phi_{\ell}''(b)|\le \Bd_{2, \ell}:=\frac{1}{\gs^2(b)\wh a_\ell^2} \gga^{-1}\big(1-\frac {2 \; \Err_\ell(b)}{A_\ell(b)}\big). 
\end{eqnarray}
\item[{\rm (b)}] For $\ell=1, 2, \cdots, K$, 
\begin{equation}
\label{comp_xk_est}
\big|U_{x}(\wh a_\ell, b, \wh \gl_\ell)- x_\ell(b)\big|
\le\Bd_{3, \ell},
\end{equation}
where 
\begin{equation*}
  \Bd_{3, \ell}:=\Err_\ell(b)+2\pi I_1 A_\ell(b) \gb^{-1}\big(1-\frac {2 \; \Err_\ell(b)}{A_\ell(b)}\big)
+\pi I_2 A_\ell(b) \gga^{-1}\big(1-\frac {2 \; \Err_\ell(b)}{A_\ell(b)}\big)  
\end{equation*}
with $I_1$ and $I_2$ defined by \eqref{def_In}.
\item[{\rm (c)}] 
 If, in addition, the window function $g(t)\ge 0$ for $t\in \RR$, then for $\ell=1, 2, \cdots, K$,
\begin{equation}
\label{abs_IA_est}
\big| |U_x(\wh a_{\ell}, b, \wh \gl_\ell)|-A_{\ell}(b) \big|\le \Err_\ell(b). 
\end{equation}
\end{enumerate}
\end{theo}

Note that since $\lim_{\xi\to 1^-}  \gb^{-1}(\xi)=0$ and $\lim_{\xi\to 1^-}  \gga^{-1}(\xi)=0$,  the error bounds $\Bd_{1, \ell}$,  $\Bd_{2, \ell}$,  $\Bd_{3, \ell}$ are small as long as $\Err_\ell(b)$ is small. We will study these error bounds in more details in the next section when $g$ is the Gaussian window function.

As shown in \eqref{phi_est}-\eqref{abs_IA_est}, $\frac \mu{\wh a_\ell(b)}$ is an estimate to $\phi'_\ell(b)$ as shown in \eqref{IF_estimate} and $U_{x}(\wh a_\ell, b, \wh \gl_\ell)$ is the recovered component of $x_\ell(b)$. 
For a real-valued $x_\ell(t)$, we will use 
\begin{eqnarray}
\label{comp_recover_real}
&&x_\ell(b)\approx 2{\rm Re}\Big(U_x(\wh a_\ell,  b, \wh \gl_\ell)\Big).  
\end{eqnarray}
In addition, the chirp rate $\phi^{\gp\gp}_\ell(b)$ and IA $A_\ell(b)$ can be estimated by $\wh \gl_\ell(b)$ and $|U_x(\wh a_{\ell}, b, \wh \gl_\ell)|$ respectively.

\begin{mrem}
Adaptive TSC-R defined by \eqref{def_adaptiveTSC-R} can be extended to adaptive CWLT with a higher order polynomial phase function. More precisely, one may define  
	\begin{eqnarray}
\nonumber	 U_x(a, b, \gl_1, \cdots, \gl_{m}) \hskip -0.6cm  && := \int_{-\infty}^\infty x(t) \frac 1{a}\overline{\psi_{\gs(b)} \big(\frac{t-b}a\big)} e^{ -i2\pi \sum_{\ell=2}^{m+1}\lambda_{\ell-1}\frac{(t-b)^\ell}{\ell!}} dt\\
	\label{def_adaptiveTSC-R_high}
	&& = \int_{\RR} x(b+at)  \frac 1 {\gs(b)} g\big(\frac t {\gs(b)}\big) e^{-i2\pi \mu  t 
	-i2\pi \sum_{\ell=2}^{m+1}\lambda_{\ell-1}\frac{(at)^\ell}{\ell!}}
	 dt. 
	\end{eqnarray}
$U_x(a, b, \gl_1, \cdots, \gl_{m})$ can be used for IF estimation and mode recovery of such a multicomponent signal that  IFs $\phi^{\gp}_k(t)$ and $\phi^{\gp}_\ell(t)$ of two components are \lq\lq{}highly\rq\rq{} crossover at some time $t_0$:  $\phi^{(j)}_k(t_0)=\phi^{(j)}_\ell(t_0), 1\le j\le m$. 
% have equal  IFs and derivatives of IFs up to certain order at some time $t_0$.   
One can establish theorems similar Theorems \ref{theo:adaptive TSC-R1} and \ref{theo:adaptive TSC-R2} for $U_x(a, b, \gl_1, \cdots, \gl_{m})$. 	
\hfill $\blacksquare$
\end{mrem}

\bigskip 
Finally in this section we present the proofs of Theorems \ref{theo:adaptive TSC-R1} and \ref{theo:adaptive TSC-R2}. For simplicity of presentation, we write $\gs$ for $\gs(b)$. 

{\bf Proof of Theorem \ref{theo:adaptive TSC-R1}(a).} Clearly $\cup_{k=1}^K \cH_{b, k}\subseteq \cH_b$. Next we show $\cH_b\subseteq \cup_{k=1}^K \cH_{b, k}$. 

Let $(a, \gl)\in \cH_b$. Suppose $(a, \gl) \not \in \cH_{b, k}$ for any $k$.  Then $(a, b, \gl) \not \in \cup_{k=1}^K Z_k$. 
Hence, by \eqref{def_upper_bounds}, we have 
%$$ |\eta -\phi'_k(t)| +\rho  \; |\gl - \phi''_k(t) |\ge \gt. $$ Hence, 
\begin{eqnarray*}
\big|\fr_x(a, b, \gl)\big| \hskip -0.6cm &&=\Big|\sum_{k=1}^K x_k(b) \wb g\big( \gs(\mu -a \phi'_k(b)), \gs^2a^2 (\gl - \phi''_k(b))\big)\Big|\\
&& \le \sum_{k=1}^K A_k(b) \gUp(b)=M(b)\gUp(b).  
\end{eqnarray*}
This, together with \eqref{S_R_error_interal}, implies 
\begin{eqnarray*}
\big|U_x(a, b, \gl)\big|\hskip -0.6cm && \le \big|U_x(a, b, \gl)-\fr_x(a, b, \gl)\big|+\big|\fr_x(a, b, \gl)\big|\\
&& \le  M(b)\Pi_0(b)+ M(b)\gUp(b) \le \wt \ep_1,   
\end{eqnarray*}
a contradiction to that $(a, \gl)\in \cH_b$. Hence there must exist an $\ell$ such that $(a, \gl)\in \cH_{b,\ell}$. This shows $\cH_b=\cup_{k=1}^K \cH_{b, k}$. 

\bigskip 
{\bf Proof of Theorem \ref{theo:adaptive TSC-R1}(b).}  Observe that $\cH_{b, k}=\cH_b\cap \{(a, \gl): (a, b, \gl)\in Z_ k\}$.  Since  $Z_k, 1\le k \le K$ are disjoint, we conclude that  $\cH_{b, k}, 1\le k \le K$ are also disjoint. 

\bigskip 
{\bf Proof of Theorem \ref{theo:adaptive TSC-R1}(c).} 
To show that each $\cH_{b, \ell}$ is non-empty,  it is enough to show $(\frac{\mu}{\phi'_\ell(b)}, \phi''_\ell(b))\in \cH_b$. Indeed,  with $\wb g(0, 0)=1$, 
\eqref{U_with_ell} with $\eta=\frac{\mu}{\phi'_\ell(b)}, \gl=\phi''_\ell(b)$ 
implies 
\begin{eqnarray*}
&&
\big|U_x(\frac{\mu}{\phi'_\ell(b)}, b, \phi''_\ell(b))\big|\ge \big|x_\ell(b) \wb g(0, 0)\big|- \Err_\ell(b)\\
&& =    A_\ell(b)-M(b)\Pi_\ell(b)-\sum_{k\not= \ell} A_k(b) \gUp_{\ell, k}(b)\\
&& > \nu(b)-M(b)\Pi_0(b)-M(b)\gUp(b)\ge \wt \ep_1. 
\end{eqnarray*}
Thus  $(\frac{\mu}{\phi'_\ell(b)}, \phi''_\ell(b))\in \cH_b$. Hence  $(\frac{\mu}{\phi'_\ell(b)}, \phi''_\ell(b))\in \cH_{b,\ell}$, and $\cH_{b,\ell}$ is non-empty. 
\hfill $\blacksquare$

\bigskip 
{\bf Proof of Theorem \ref{theo:adaptive TSC-R2}(a).} 
From \eqref{U_with_ell}, we have 
\begin{equation}
\label{est_xk1}
\big|U_x(\wh a_\ell, b, \wh \gl_\ell)\big| \le 
\big|x_\ell(b) \wb g\big(\gs(\mu-\wh a_\ell \phi_\ell'(b)), \gs^2\wh a_\ell^2(\wh \gl_\ell-\phi_\ell''(b)) \big)\big|+ \Err_\ell(b). 
\end{equation}
On the other hand, by the definitions of $\wh a_\ell, \wh \gl_\ell$ and by \eqref{U_with_ell} with $a=\frac\mu{\phi_\ell'(b)}, \gl=\phi_\ell''(b)$, we have 
\begin{equation}
\label{est_xk2}
\big|U_x(\wh a_\ell, b, \wh \gl_\ell)\big|  \ge \big|U_x(\frac\mu{\phi_\ell'(b)}, b, \phi_\ell''(b))\big |\ge |x_\ell(b) \wb g(0, 0) \big| - \Err_\ell(b)
=A_\ell(b)-   \Err_\ell(b). 
\end{equation}
This, together with \eqref{est_xk1}, implies 
$$
A_\ell(b) - \Err_\ell(b)\le A_\ell(b) 
\big|\wb g\big(\gs(\mu-\wh a_\ell \phi_\ell'(b)), \gs^2\wh a_\ell^2(\wh \gl_\ell-\phi_\ell''(b)) \big)\big|
+ \Err_\ell(b). 
$$
Thus we have 
\begin{equation}
\label{est_xk3}
1- \frac{2\; \Err_\ell(b)}{A_\ell(b)}\le  f\big(\gs|\mu-\wh a_\ell \phi_\ell'(b)|, \gs^2 \wh a_\ell^2 |\wh \gl_\ell-\phi_\ell''(b)| \big).  
\end{equation}
Since $2\Err_\ell(b)/A_\ell(b)\le c_0$,  \eqref{est_xk3} along with \eqref{ineq_cond} and \eqref{cond_gb_gga}  leads to 
$$
1- \frac{2\; \Err_\ell(b)}{A_\ell(b)}\le \gb \big(\gs\big|\mu-\wh a_\ell \phi_\ell'(b)\big|\big), \quad 
1- \frac{2\; \Err_\ell(b)}{A_\ell(b)}\le \gga \big(\gs^2\wh a_\ell^2 \big|\wh \gl_\ell-\phi_\ell''(b)\big|\big).  
$$
Since $\gb(\xi), \gga(\xi)$ decreasing, we have 
$$
\gs\big| \mu-\wh a_\ell \phi_\ell'(b) \big|\le \gb^{-1}\big(1- \frac{2\; \Err_\ell(b)}{A_\ell(b)}\big), \;
\gs^2\wh a_\ell^2\big| \wh \gl_\ell-\phi_\ell''(b) \big|\le \gga^{-1}\big(1- \frac{2\; \Err_\ell(b)}{A_\ell(b)}\big). 
$$
Thus shows \eqref{phi_est} and \eqref{phi_est_gl}.

\bigskip 

{\bf Proof of Theorem \ref{theo:adaptive TSC-R2}(b).} 
From \eqref{U_with_ell}, we have 
\begin{eqnarray*}
&&\big|U_x(\wh a_\ell, b, \wh\gl_\ell)- x_\ell(b)\big|\le \big|U_x(\wh a_\ell, b, \wh\gl_\ell) - x_\ell(b) \wb g\big( \gs(\mu - \wh a_\ell \phi'_\ell(b)), \gs^2\wh a_\ell^2(\wh \gl_\ell - \phi''_\ell(b))\big)\big|\\
&& \qquad +\big |x_\ell(b) \wb g\big( \gs(\mu - \wh a_\ell \phi'_\ell(b)), \gs^2\wh a_\ell^2(\wh \gl_\ell - \phi''_\ell(b))\big)
- x_\ell(b)\big|
\\
&&\le \Err_\ell(b)+A_\ell(b) \Big| \int_\RR \frac 1\gs g(\frac t\gs)\Big(e^{-i2\pi (\mu - \wh a_\ell \phi'_\ell(b))t -i\pi \wh a_\ell^2(\wh \gl_\ell- \phi_\ell''(b)) t^2}-1\Big) dt \Big|\\
&&\le \Err_\ell(b)+A_\ell(b) \int_\RR \Big| \frac 1\gs g(\frac t\gs)\Big| \; \big|2\pi (\mu - \wh a_\ell \phi'_\ell(b))t +\pi\wh a_\ell^2 (\wh \gl_\ell- \phi_\ell''(b)) t^2 \big| dt \\
&&\le \Err_\ell(b)+A_\ell(b) 2\pi |\mu - \wh a_\ell \phi'_\ell(b)|\int_\RR \Big| \frac 1\gs g(\frac t\gs) t \Big| dt +A_\ell(b)\pi\wh a_\ell^2 \big|\wh \gl_\ell- \phi_\ell''(b)\big| \int_\RR \frac 1\gs \Big|g(\frac t\gs)\Big|  t^2 dt \\
&&= \Err_\ell(b)+A_\ell(b) 2\pi I_1 \gs |\mu-\wh a_\ell\phi_\ell'(b)|+ A_\ell(b)\pi I_2 \gs^2\wh a_\ell^2
\big|\wh \gl_\ell- \phi_\ell''(b)\big|\\
&&\le  \Err_\ell(b)+2\pi I_1 A_\ell(b) \gb^{-1}\big(1-\frac {2 \; \Err_\ell(b)}{A_\ell(b)}\big)+
\pi I_2 A_\ell(b) \gga^{-1}\big(1-\frac {2 \; \Err_\ell(b)}{A_\ell(b)}\big), 
\end{eqnarray*}
where the last inequality follows from \eqref{phi_est} and \eqref{phi_est_gl}. 
This completes the proof of \eqref{comp_xk_est}. 

\bigskip 
{\bf Proof of Theorem \ref{theo:adaptive TSC-R2}(c).}  Note that when $g(t)\ge 0$, by the assumption $\int_\RR g(t)dt=1$, we have that $|\wb g(\eta, \gl)|\le 1$ for any $\eta, \gl\in \R$. This fact, together with 
\eqref{est_xk1}, implies 
\begin{equation*}
%\label{est_xk1}
\big|U_x(\wh a_\ell, b, \wh \gl_\ell)\big| \le 
A_\ell(b) + \Err_\ell(b). 
\end{equation*}
This and \eqref{est_xk2} lead to \eqref{abs_IA_est}. This completes the proof  of  Theorem \ref{theo:adaptive TSC-R2}(c).
\hfill $\blacksquare$ 

\section{Time, scale and chirp\_rate  signal recovery operator with Gaussian window function} 
The Gaussian function is the only function (up to scalar multiplication, shift and  modulations) which gains the optimal time-frequency resolution. Hence it has been used in many applications. 
In this section we consider the adaptive TSC-R with the window function being the Gaussian function 
and obtain more precise estimates for the error bounds  $\Bd_{1, \ell}$,  $\Bd_{2, \ell}$,  $\Bd_{3, \ell}$ in Theorem 2.  
In the following $g$ is always the Gaussian function given in \eqref{def_g}. 

From \eqref{g_PFT}, we have that 
 $|\wb g(\eta, \gl)|= f(|\eta|, |\gl|)$ with 
 \begin{equation}
 \label{def_f}
 f(\eta, \gl):=\frac 1{(1+4\pi^2\gl^2)^{1/4}} e^{-\frac{2\pi^2 \eta ^2}{1+4\pi^2 \gl^2}} \; .  
 \end{equation}
First one can obtain that  
\begin{equation}
\label{ineq_gaussian}
 |\wb g(\eta, \gl)|\le\min\Big\{\frac 1{(2\pi^2 \eta ^2)^{1/4}}, 
\frac 1{(1+4\pi^2\gl^2)^{1/4}}\Big\}. 
\end{equation}
Indeed, if 
$2\pi^2 \eta ^2\ge 1+4\pi^2 \gl^2$, then 
\begin{eqnarray*}
|\wb g(\eta, \gl)|\hskip -0.6cm &&\le \frac 1{(1+4\pi^2\gl^2)^{1/4}} \frac{1+4\pi^2 \gl^2}{2\pi^2 \eta ^2}
=\frac{(1+4\pi^2 \gl^2)^{3/4}}{2\pi^2 \eta ^2}
\le \frac 1{(2\pi^2 \eta ^2)^{1/4}};
\end{eqnarray*}
otherwise, for  
$2\pi^2 \eta ^2< 1+4\pi^2 \gl^2$, we have  
\begin{eqnarray*}
|\wb g(\eta, \gl)|\hskip -0.6cm &&\le \frac 1{(1+4\pi^2\gl^2)^{1/4}}.
\end{eqnarray*}
Hence \eqref{ineq_gaussian} holds. 

Next let us consider the quantities $\Upsilon(b)$ and $\Upsilon_{\ell, k}(b)$ satisfying  \eqref{def_upper_bounds}. 
Suppose $(a, b, \gl)\not \in Z_k$. By \eqref{ineq_gaussian}, we have 
\begin{eqnarray}
\nonumber &&\big|\wb g\big( \gs (\mu -a \phi'_k(b)), \gs^2a^2 (\gl - \phi''_k(b))\big)\big| \\
\label{Ineq_Gaussian1} 
&&\qquad  \le \min\Big\{\frac 1{(2\pi^2)^{1/4} \sqrt{\gs |\mu -a \phi'_k(b)|}}, 
\frac 1{\big(1+4\pi^2 \gs^4a^4 (\gl - \phi''_k(b))^2 \big)^{1/4}}\Big\}. 
\end{eqnarray}
If $|\mu -a \phi'_k(b)|\ge \gt$, then 
$$
\frac 1{(2\pi^2)^{1/4} \sqrt{\gs |\mu -a \phi'_k(b)|}}\le \frac 1{(2\pi^2)^{1/4}\sqrt \gt  \sqrt{\gs} }; 
$$
otherwise,  if $|\mu -a \phi'_k(b)|<\gt$, then $|\gl - \phi''_k(b))|\ge  \gt_1$.  Therefore, 
$$
\frac 1{\big(1+4\pi^2 \gs^4a^4 (\gl - \phi''_k(b))^2 \big)^{1/4}}\le \frac 1{\sqrt{2\pi \gt_1} a\gs} 
\le \frac 1{\sqrt{2\pi \gt_1} a_1 \gs}. 
$$
Hence, by \eqref{Ineq_Gaussian1}, we have  
\begin{eqnarray}
 \label{Ineq_Gaussian2} &&\big|\wb g\big( \gs (\mu -a \phi'_k(b)), \gs^2a^2 (\gl - \phi''_k(b))\big)\big| 
\le \max\Big\{\frac 1{(2\pi^2)^{1/4}\sqrt \gt  \sqrt{\gs} }, \frac 1{\sqrt{2\pi \gt_1} a_1 \gs}\Big\}. 
 \end{eqnarray}
 Thus we may let 
 $$
 \Upsilon(b)=\frac 1{\sqrt{\gs}\min \big\{(2\pi^2)^{1/4}\sqrt \gt, a_1\sqrt{2\pi \gt_1 \gs} \big\} }.
 $$
 Since $Z_\ell$ and $Z_k$ are not overlapping if $\ell\not= k$, we may simply let $\Upsilon_{\ell, k}(b)=\Upsilon(b)$. 
 For such choice of $\Upsilon(b)$ and $\Upsilon_{\ell, k}(b)$,  \eqref{def_upper_bounds} holds. Note that if $\gs=\gs(b)$ is large, 
 then $\Upsilon(b)$ and $\Upsilon_{\ell, k}(b)$ will be small. 
 
\bigskip 
Next we consider the functions $\gb(\xi), \gga(\xi)$ satisfying \eqref{cond_gb_gga} for $f(\eta, \gl)$ given by 
\eqref{def_f}. Clearly, we may choose 
 \begin{equation}
\label{def_gga}
\gga(\gl)=\frac 1{(1+4\pi^2\gl^2)^{1/4}}. 
\end{equation}
Next we will show that for this $f(\eta, \gl)$, if $c_0$ in \eqref{ineq_cond} satisfies 
$c_0\le 1-e^{-1/4}$, then we can choose 
\begin{equation}
\label{def_gb}
\gb(\eta)=e^{-2\pi^2 \eta ^2}. 
\end{equation}
To this regard, we first have the following two lemmas. 

\begin{lem}\label{lem4}
Let  $f(\eta, \gl)$ be the function defined by \eqref{def_f}. If $0\le\eta\le \frac 1{2\pi \sqrt 2}$, then  
\begin{equation}
\label{ineq_f_1}
f(\eta, \gl)\le f(\eta, 0)=e^{-2\pi^2 \eta ^2}, \; \gl \in [0, \infty). 
\end{equation}
\end{lem}
\begin{proof}
One can obtain from $\partial_\gl f(\eta, \gl)$ that for fixed $\eta$, $f(\eta, \gl)$ is a decreasing function in $\gl$ on $\gl\ge 0$ and 
\begin{equation}
\label{small_ineq}
8\pi^2 \eta^2\le 1+4\pi^2 \gl^2.
\end{equation}
Notice that \eqref{small_ineq} holds true for any $\gl\ge 0$ if $0\le\eta\le \frac 1{2\pi \sqrt 2}$. Hence,  \eqref{ineq_f_1} holds true. 
 \end{proof}

\begin{lem}\label{lem5}
 Let $f(\eta, \gl)$ be the function defined by \eqref{def_f}. Let $c$ be a number satisfying $0\le c\le 1-e^{-1/4}$. Then 
 $1-c\le f(\eta, \gl)$ for some $\eta, \gl\ge 0$ implies $0\le\eta\le \frac 1{2\pi \sqrt 2}$. 
 \end{lem}
 \begin{proof}
 Assume $\eta > \frac 1{2\pi \sqrt 2}$. Let $\gl\ge 0$. If $1+4\pi ^2 \gl^2 \le 8\pi ^2 \eta^2$, then 
 $$
 f(\eta, \gl)\le \frac 1{(1+4\pi^2\gl^2)^{1/4}} e^{-1/4} <e^{-1/4}. 
 $$
Otherwise, when $1+4\pi ^2 \gl^2 > 8\pi ^2 \eta^2$, let $\gl_0>0$ be the number such that 
$1+4\pi ^2 \gl_0^2 =8\pi ^2 \eta^2$. As mentioned in the proof of Lemma \ref{lem4}, $f(\eta, \gl)$ is a decreasing function in $\gl$ for $8\pi^2 \eta^2\le 1+4\pi^2 \gl^2$. Since $\gl_0<\gl$, we have 
$$f(\eta, \gl)\le f(\eta, \gl_0)=\frac 1{(8\pi^2 \eta^2)^{1/4}}e^{-1/4}<e^{-1/4}. 
$$
So in either case, we have $f(\eta, \gl)<e^{-1/4}$, a contradiction to that 
$$
f(\eta, \gl)\ge 1-c\ge e^{-1/4}. 
$$
 Therefore $\eta \le \frac 1{2\pi \sqrt 2}$. This completes the proof of Lemma \ref{lem5}. 
 \end{proof}

Lemmas \ref{lem4} and \ref{lem5} immediately lead to the following proposition. 
\begin{pro}
 Let $f(\eta, \gl)$, $\gb(\eta)$ and $\gga(\gl)$ be the functions defined by \eqref{def_f}, \eqref{def_gb} and \eqref{def_gga} respectively. Suppose $c$ satisfies $0\le c\le 1-e^{-1/4}$. Then  $1-c\le f(\eta, \gl)$ implies  
 $$
 1-c \le \gb(\eta), \; 1-c\le \gga(\gl). 
 $$
 \end{pro}
\begin{proof} Clearly $1-c\le \gga(\gl)$ since $f(\eta, \gl)\le \gga(\gl)$.
By Lemma \ref{lem5}, $1-c\le f(\eta, \gl)$ implies $\eta \le \frac 1{2\pi \sqrt 2}$. This, together with Lemma \ref{lem4}, implies 
$$
f(\eta, \gl)\le f(\eta, 0)= \gb(\eta).
$$
Thus $1-c\le f(\eta, \gl)\le  \gb(\eta)$, as desired.
 \end{proof}

 For $\gga(\gl)$ given by \eqref{def_gga}, its inverse $\gga^{-1}(\xi)$ is given by 
 $$
 \gga^{-1}(\xi)=\frac 1{2\pi \xi^2}\sqrt{1-\xi^4}. 
 $$
Hence the error bound $\Bd_{2, \ell}$ in \eqref{phi_est_gl} is given by 
\begin{eqnarray} 
\nonumber \Bd_{2, \ell}\hskip -0.6cm &&:=\frac{1}{\gs^2(b)} \gga^{-1}\big(1- \frac{2\; \Err_\ell(b)}{A_\ell(b)}\big)\\
\label{B2_est} &&=
 \frac 1{\gs^2(b) 2\pi \big(1-\frac{2 \; \Err_\ell(b)}{A_\ell(b)}\big)^2}\sqrt{1-\big(1-\frac{2 \; \Err_\ell(b)}{A_\ell(b)}\big)^4} \; . 
\end{eqnarray}
Hence, if  $\frac{\Err_\ell(b)}{A_\ell(b)}\approx 0$, then 
$$
\Bd_{2, \ell}\approx \frac {\sqrt 2}{\pi\gs^2(b)} \sqrt{\frac{\Err_\ell(b)}{A_\ell(b)}} \; .
$$

\bigskip 
The inverse function $\gb^{-1}(\xi)$ of $\gb(\gl)$ given by \eqref{def_gb} is 
 $$
 \gb^{-1}(\xi)=\frac1{\pi \sqrt 2} \sqrt{-\ln \xi}, \;  0<\xi <1. 
 $$
Thus if $\frac{2 \; \Err_\ell(b)}{A_\ell(b)}\le 1-e^{-1/4}$, 
then by \eqref{est_xk3} and Proposition 1, we have 
\begin{eqnarray*}
&&|\mu - \wh a_{\ell}(b)\phi_{\ell}'(b)|\le \Bd_{1, \ell}:=\frac{1}{\gs(b)} \gb^{-1}\big(1-\frac {2 \; \Err_\ell(b)}{A_\ell(b)}\big)\\
&& = \frac{1}{\gs(b) {\pi \sqrt 2}} \sqrt{-\ln \big(1-\frac {2 \; \Err_\ell(b)}{A_\ell(b)}\big)}.
\end{eqnarray*}
Using the fact $-\ln(1-t)< e^{1/4}t$ for $0<t<1-e^{-1/4}$, we have 
\begin{equation}
\label{B1_est}
\Bd_{1, \ell}\le \frac{e^{1/8}}{\gs(b)\pi } \sqrt{\frac{\Err_\ell(b)}{A_\ell(b)}}. 
\end{equation}
In addition, the error bound $\Bd_{3, \ell}$ in \eqref{comp_xk_est} for component recovery is bounded by 
\begin{equation}
\label{B3_est}
  \Bd_{3, \ell}\le \Err_\ell(b)+2 e^{1/8} I_1 \sqrt{{\Err_\ell(b)}{A_\ell(b)}}
+\frac{I_2 A_\ell(b)}{2\big(1-\frac{2 \; \Err_\ell(b)}{A_\ell(b)}\big)^2}\sqrt{1-\big(1-\frac{2 \; \Err_\ell(b)}{A_\ell(b)}\big)^4}\; . 
\end{equation}

To summarize, we have the following theorem. 
\begin{theo}\label{theo:adaptive TSC-R3}
Let $x(t)\in \cE_{\ep_1,\ep_3}$ for some $\ep_1>0,\ep_3>0$, and $U_x(a, b, \gl)$ be the adaptive TSC-R 
of $x(t)$ with Gaussian window function $g$ in \eqref{def_g}.  Suppose \eqref{def_sep_cond_cros} and \eqref{theo1_cond1} hold and that for $1\le \ell \le K$, $2\Err_\ell(b)/A_\ell(b)\le 1-e^{-1/4}$. 
Let $\cH_b$ and $\cH_{b, k}$ be the sets defined by \eqref{def_cGk} for some 
$\wt \ep_1$ satisfying  \eqref{cond_ep1}. Let $\wh a_\ell(b), \wh \gl_\ell(b)$ be the functions defined by \eqref{def_max_eta}. Then \eqref{phi_est}, \eqref{phi_est_gl} and \eqref{comp_xk_est} hold with 
$\Bd_{1, \ell}, \Bd_{2, \ell}$ and $\Bd_{3, \ell}$ bounded by the quantities in \eqref{B1_est}, \eqref{B2_est} and \eqref{B3_est} respectively. 
\end{theo}

\section{Experiments} 

%%%%%%%%%%%%%%%%%%%the beginning of figure 1 %%%%%%%%%%%%%%%
\begin{figure}[th]
	\centering
	\hspace{-0.7cm}
	\begin{tabular}{c@{\hskip -0.2cm}c @{\hskip -0.2cm}c}
		\resizebox{2.2in}{1.65in}{\includegraphics{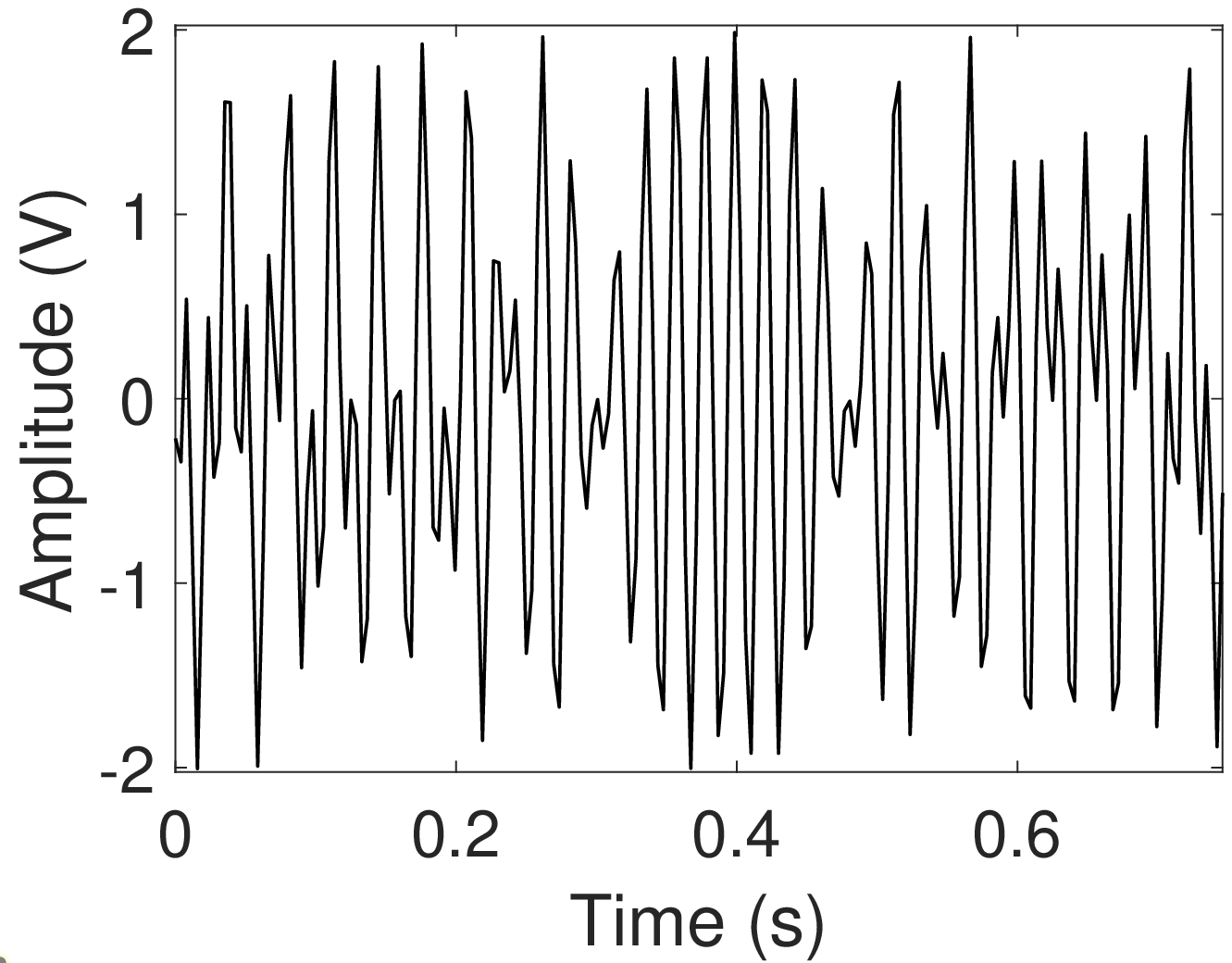}} \quad &
		\resizebox{2.2in}{1.65in}{\includegraphics{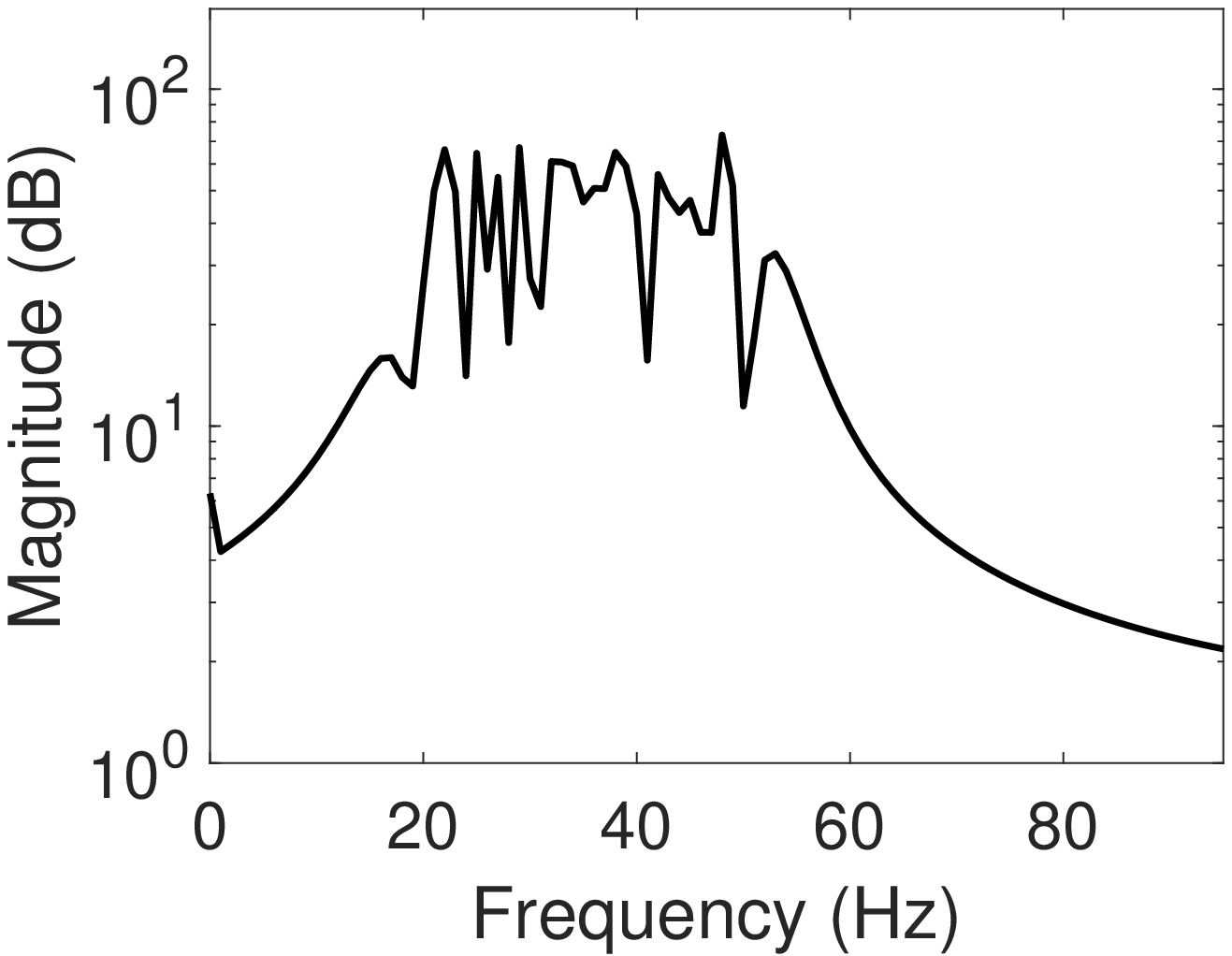}}\quad & 
		\resizebox{2.2in}{1.65in}{\includegraphics{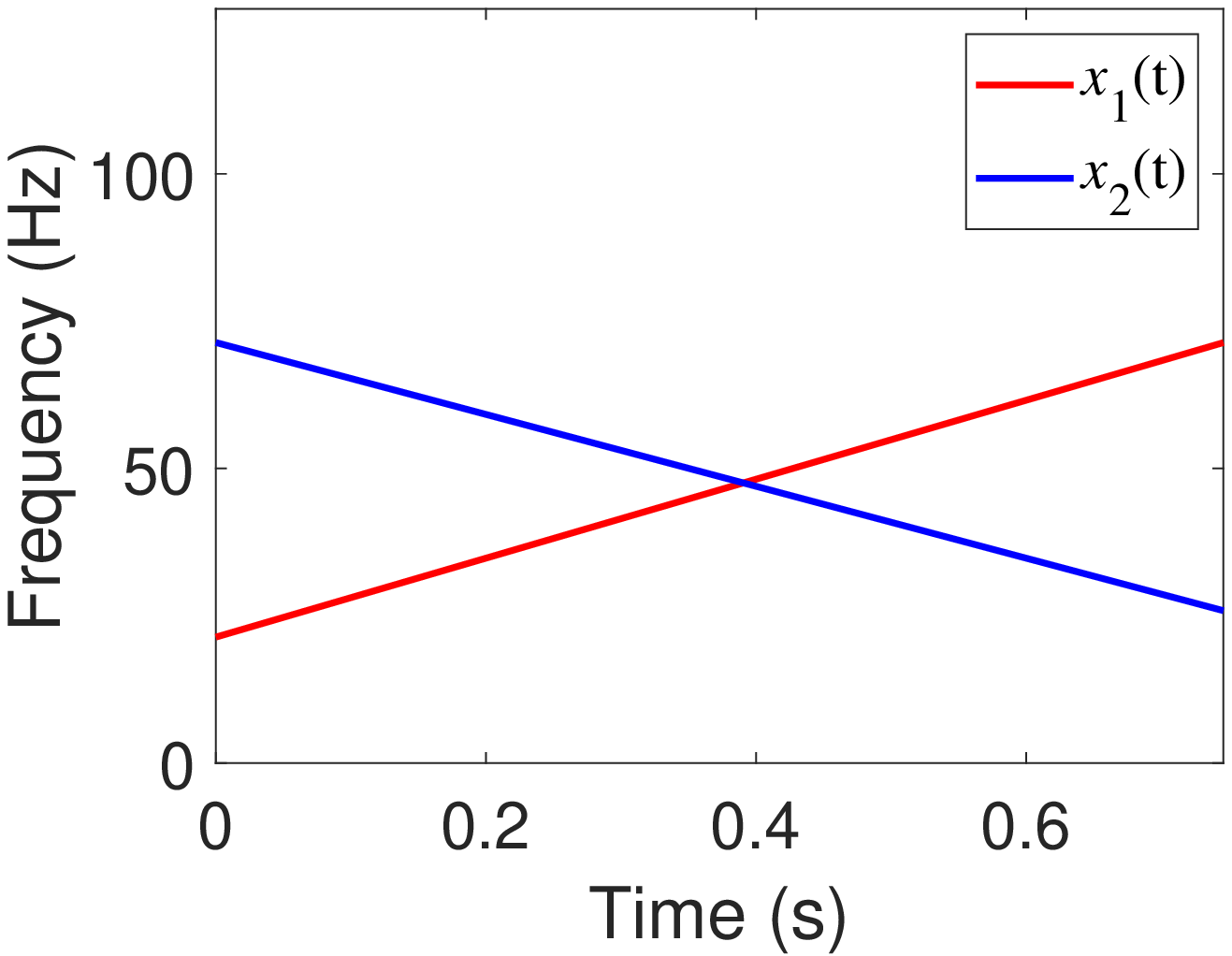}}\\ 
		\resizebox{2.2in}{1.65in}{\includegraphics{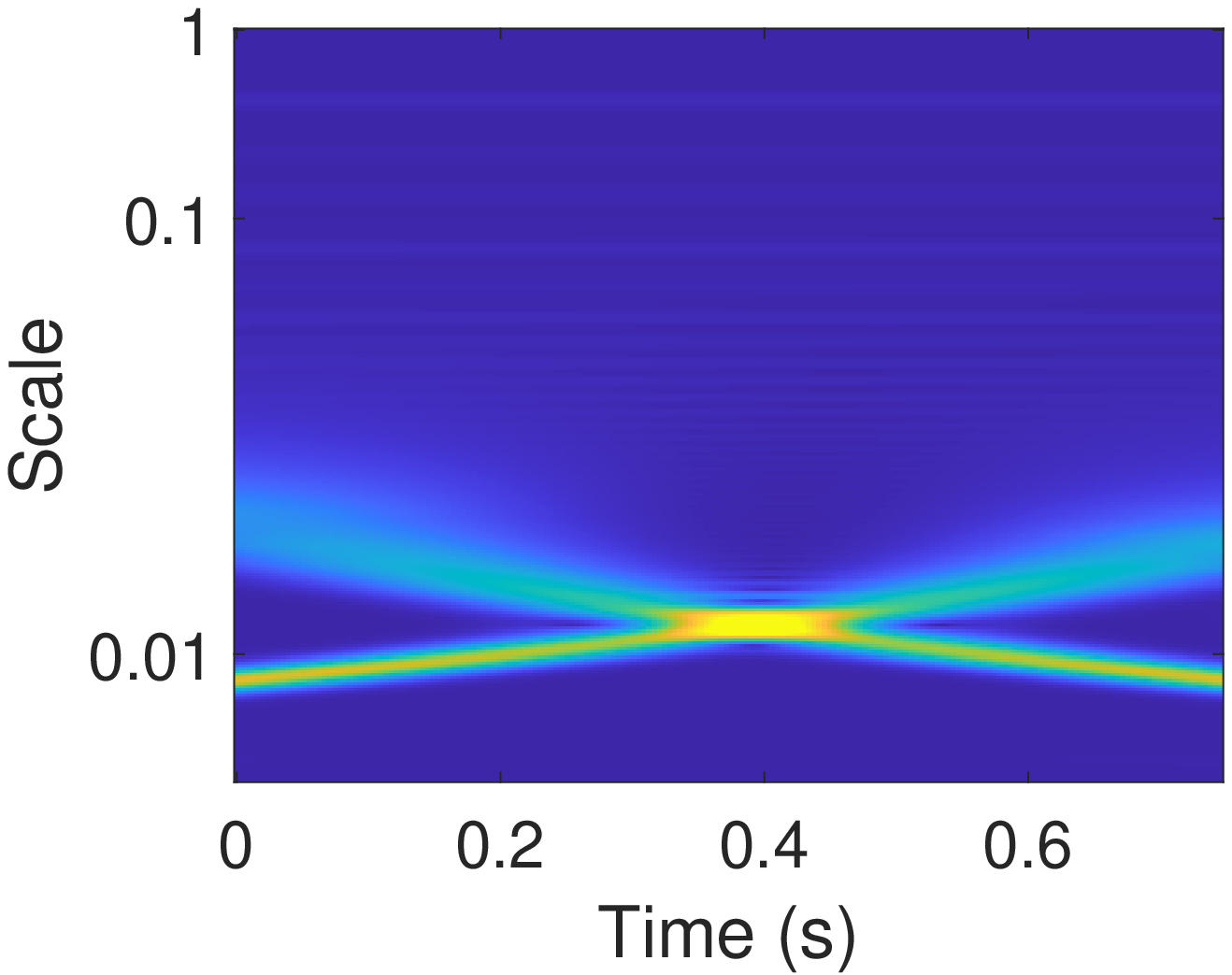}}\quad & 
		\resizebox{2.2in}{1.65in}{\includegraphics{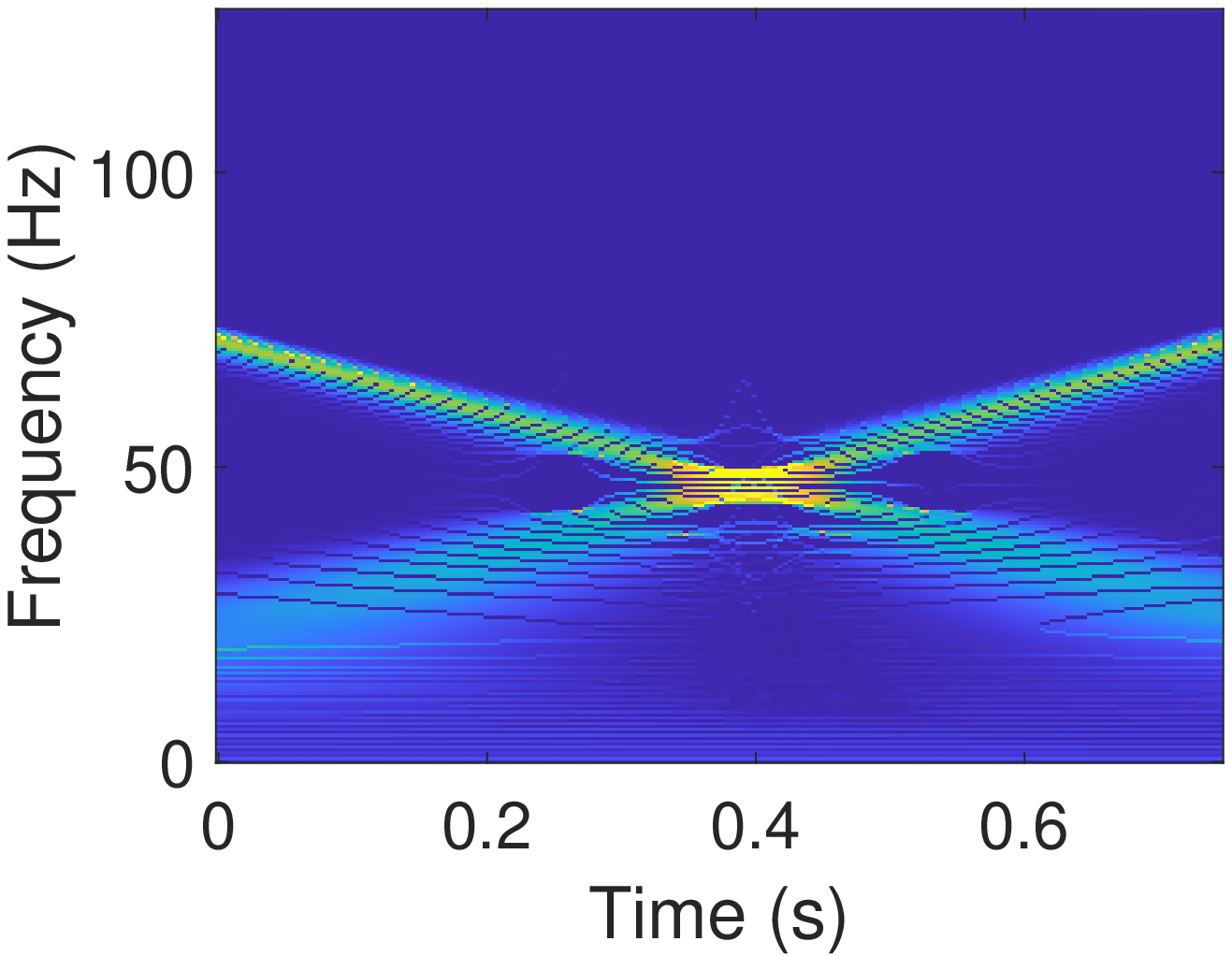}} \quad & 
		\resizebox{2.2in}{1.65in}{\includegraphics{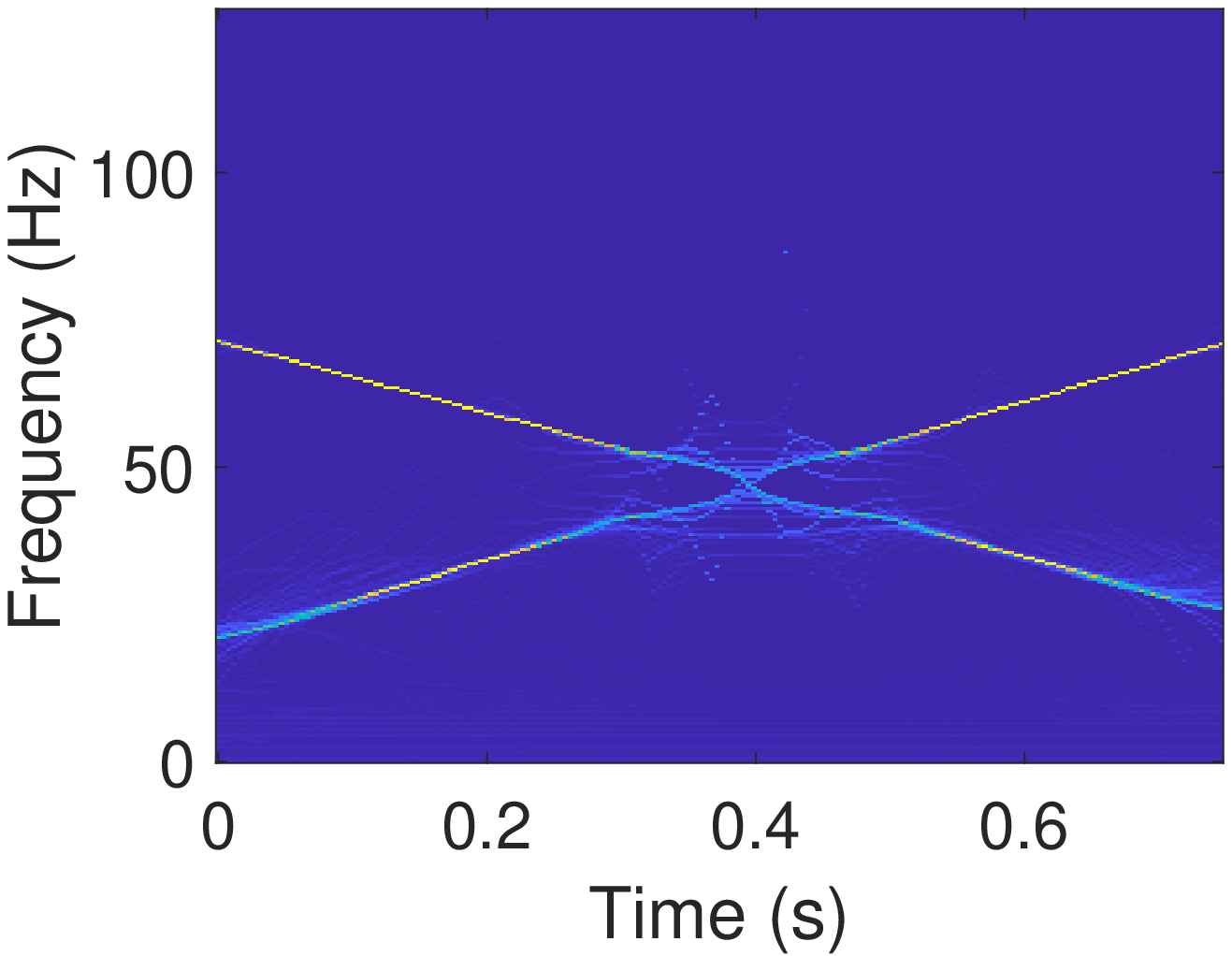}}
	\end{tabular}
	\caption{\small Two-component signal $x(t)$ in \eqref{two_lfm} and its time-frequency representations with SST. 
		Top row (from left to right):  Waveform of $x(t)$, magnitude spectrum and ground truth IFs of two components $x_1(t)$ and $x_2(t)$; Bottom row (from left to right): CWLT, CWT-based SST and CWT-based second-order SST.}
	\label{fig:two_chirp_signal}
\end{figure}
%%%%%%%%%%%%%%%%the end of figure 1  %%%%%%%%%%%%%%%%%%%%%

In this section we provide some experimental results to demonstrate our method and general theory. We set $\mu=1$. 

\begin{example}
{\rm Let $x(t)$ be a signal consisting of two-component linear chirps, given as 
\begin{equation}
\label{two_lfm}
x(t) = x_1(t)+x_2(t) = \cos(2\pi c_1 t+ \pi r_1 t^2) + \cos(2\pi c_2 t + \pi r_2 t^2), \; t\in[0,0.75), 
\end{equation}
where $c_1 = 21$, $c_2 = 71$, $r_1 = 67$ and $r_2 = -61$. }
\end{example}

The IFs of $x_1$ and  $x_2$ are $\phi_1'(t) = c_1+r_1t$ and $\phi_2'(t) = c_2+r_2t$, respectively. See the top-right panel of Figure 2.  
The chirp rates of $x_1$ and  $x_2$ are $\phi_1''(t) = r_1$ and $\phi_2''(t) = r_2$, respectively.
Here signal $x(t)$ is discretized with sampling rate 256Hz.  
That means there are 192 samples for $x(t)$. In the following, we just use these 192 samples to analyze the signal.  The waveform of $x(t)$ and its magnitude spectrum are presented in the top row of Fig.\ref{fig:two_chirp_signal}. 

The bottom row of Fig.\ref{fig:two_chirp_signal} shows the results of CWLT, SST \cite{Daub_Lu_Wu11} and the second-order SST \cite{OM17}, where parameter $\sigma=0.023$. Here the scale variable $a$ is discretized as $\big(2^{j/n_v} \gt t\big)_j$, where $\gt t=1/256$ for this example,  and $n_v$ is the number of voice. Here and below we set $n_v=32$. 
Due to the IF curves of the components are crossover, these methods cannot represent the synthetic signal sharply and separately. In addition, EMD performs poorly in decompose this signal. Consequently, these methods are hardly to recover this two-component signal with crossover IFs.

%%%%%%%%%%%%%%%%%%%the beginning of figure 2 %%%%%%%%%%%%%%%
\begin{figure}[th]
	\centering
	\hspace{-0.7cm}
	%\begin{tabular}{c@{\hskip -0.2cm}c }
	\begin{tabular}{cc }
			\resizebox{2.2in}{1.65in}{\includegraphics{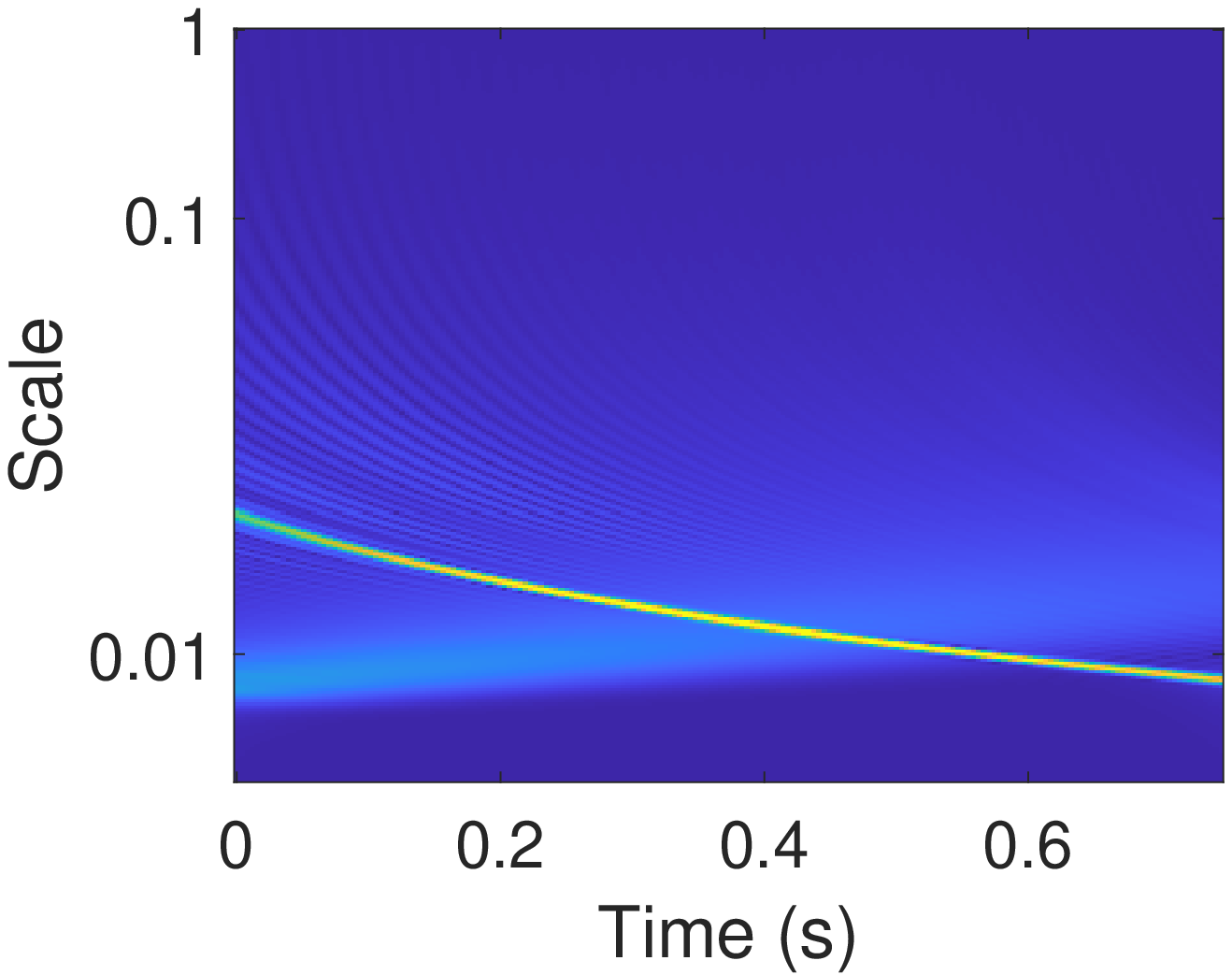}} \quad &
		\resizebox{2.2in}{1.65in}{\includegraphics{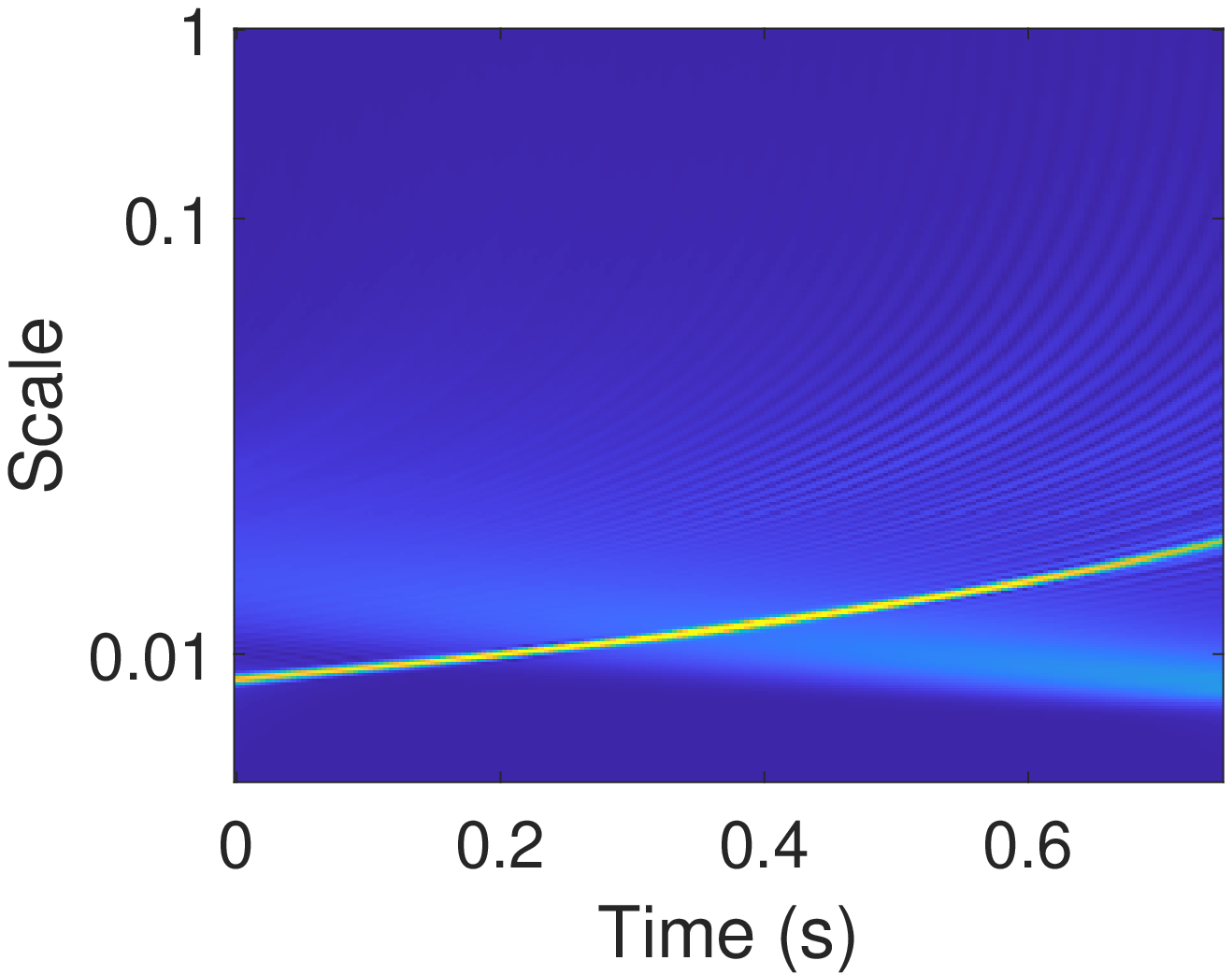}}\\ 
		\resizebox{2.2in}{1.65in}{\includegraphics{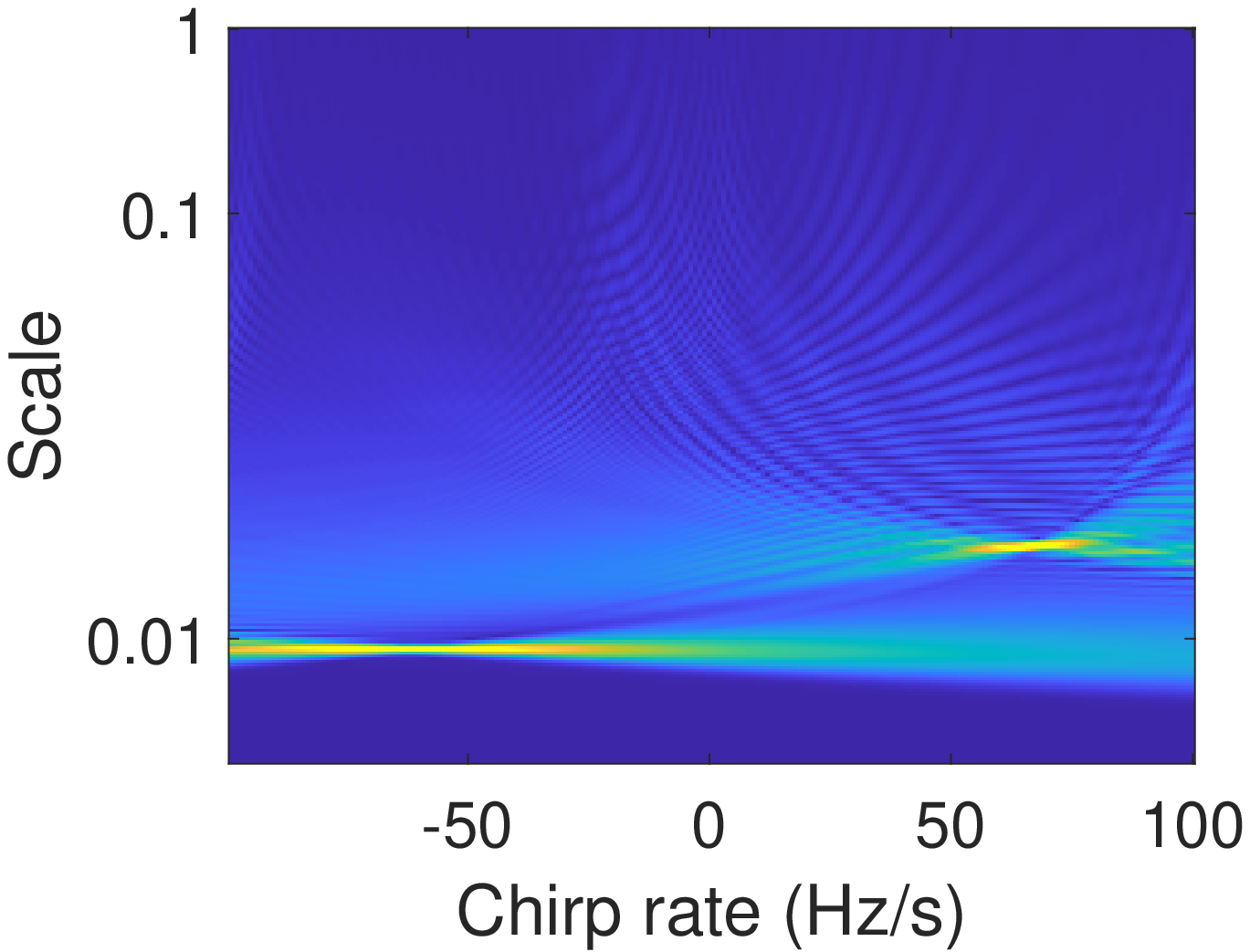}} \quad & 
		\resizebox{2.2in}{1.65in}{\includegraphics{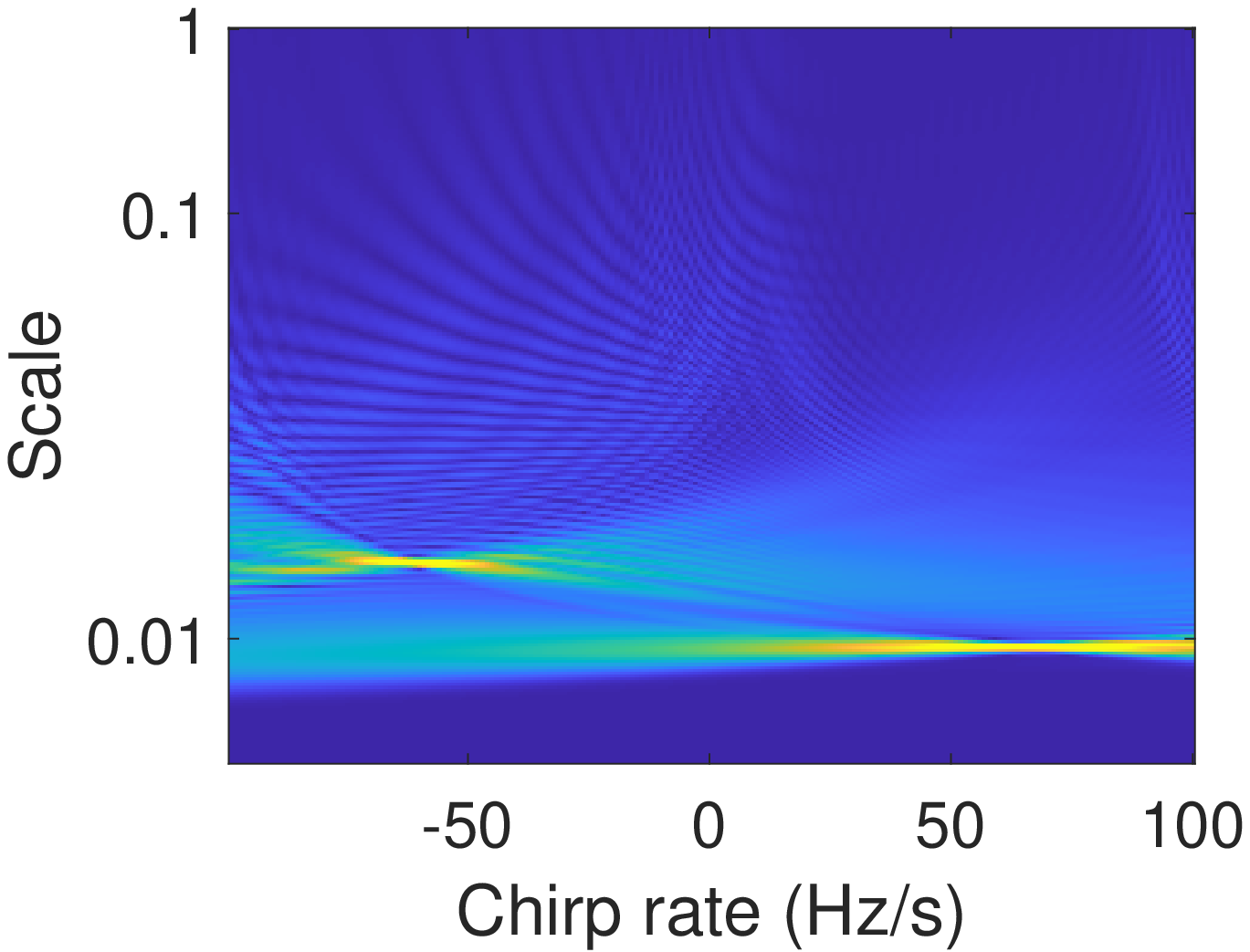}}\\
		\resizebox{2.2in}{1.65in}{\includegraphics{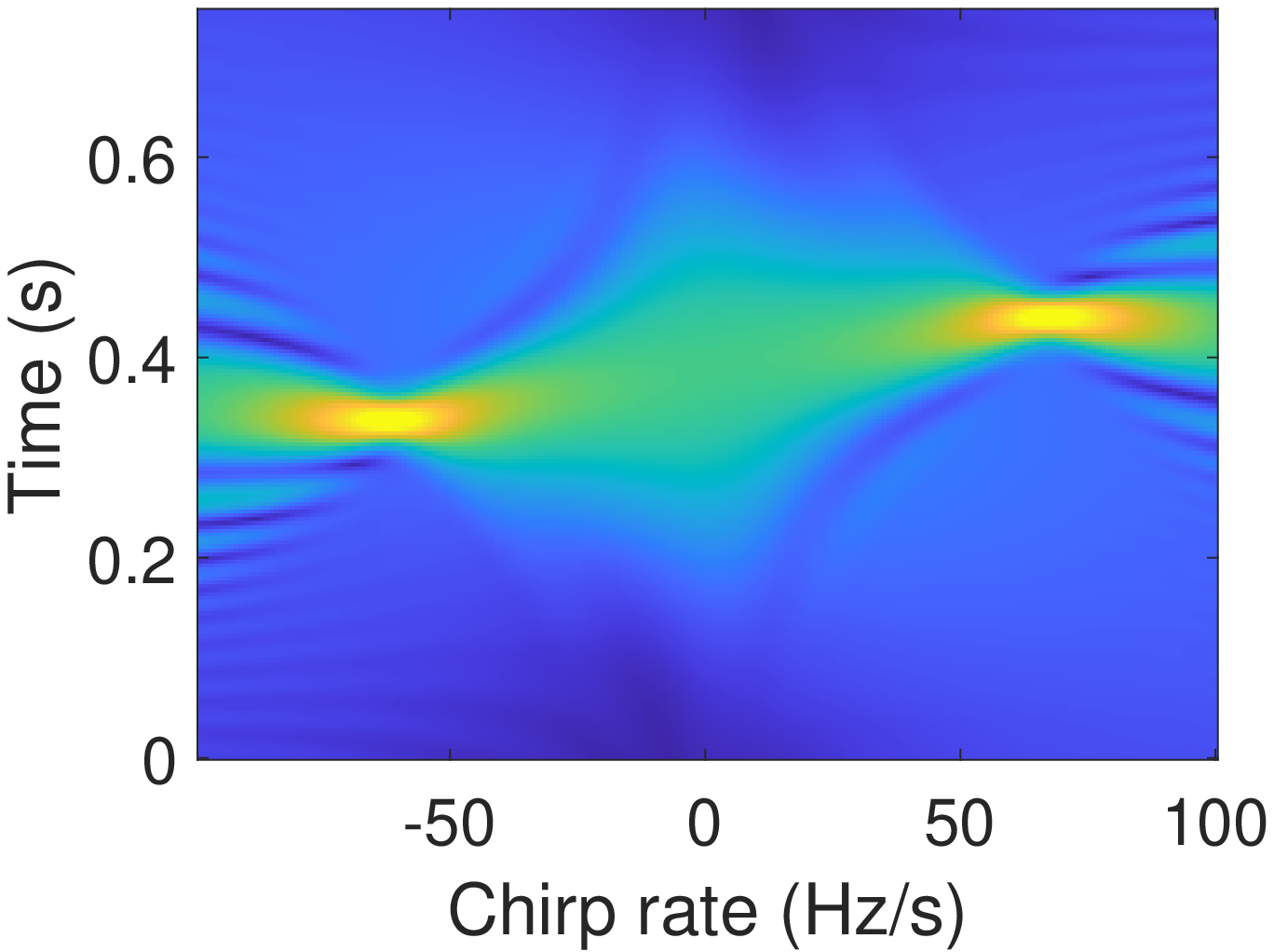}}\quad & 
		\resizebox{2.2in}{1.65in}{\includegraphics{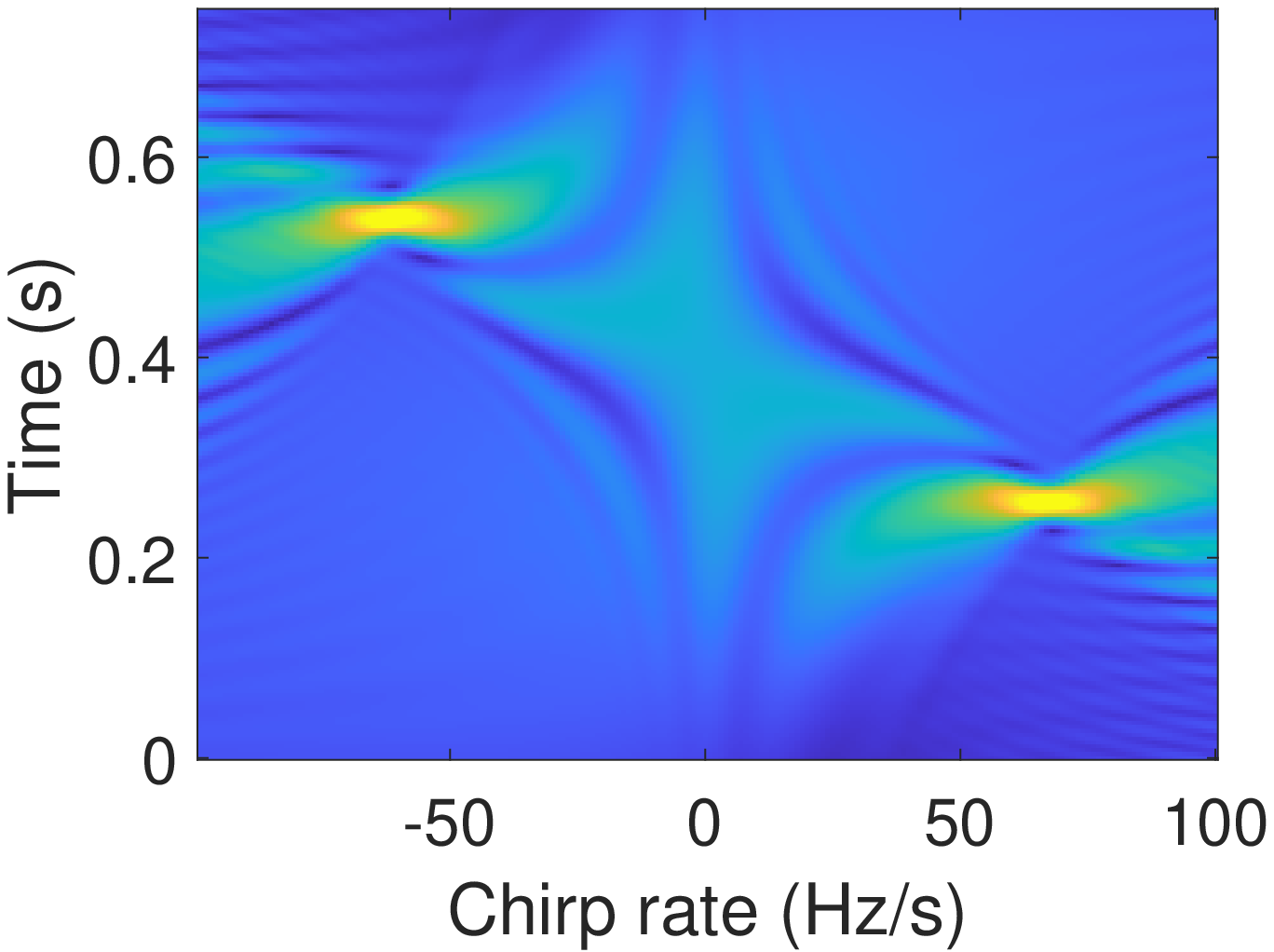}}	
	\end{tabular}
	\caption{\small Some slices of adaptive TSC-R.  
		Top row (from left to right): Two slices of $|U_x(a,b,\lambda)|$ when $\lambda=r_1$, $\lambda=r_2$;
		 Middle row (from left to right): Two slices of $|U_x(a,b,\lambda)|$ when $b=\frac{32}{256}$ and $b= \frac{160}{256}$; 
		  Bottom row (from left to right): Two slices of $|U_x(a,b,\lambda)|$ when  $a = 2^{\frac{51}{32} }/256$  and $a = \frac{1}{64}$.}
	\label{fig:two_chirp_QP_CWT}
\end{figure}
%%%%%%%%%%%%%%%%the end of figure 2 %%%%%%%%%%%%%%%%%%%%%

Next let us look at our method. Due to that the adaptive TSC-R $U_x(a,b,\lambda)$ is 3-dimensional, here we show some slices of $|U_x(a,b,\lambda)|$. First we look at the slice when $\lambda=r_1$, the ground truth chirp rate of $x_1$. The top-left panel of  Fig.\ref{fig:two_chirp_QP_CWT} is $|U_x(a,b,r_1)|$.  The clear and sharp scale-time ridge shown in this panel is exactly the curve $(b, \frac{\mu}{\phi^\gp_1(b)})$, which gives a  precise estimate of $\phi^\gp_1(b)$, the IF of $x_1(t)$. The top-right panel is $|U_x(a,b,r_2)|$, where the clear and sharp scale-time curve corresponds to $(b, \frac{\mu}{\phi^\gp_2(b)})$. These two pictures tell us that in two scale-time planes (sub-spaces of $\R^3$) $(b, a, r_1)$ and $(b, a, r_2), b, a\in \R, a>0$, 
there do exist two clear and sharp scale-time ridges which are desired to estimate $\phi^\gp_1(b)$  and $\phi^\gp_2(b)$. Note that these two scale-time planes are well-separated in the 3-dimensional space $\R^3$ since the distance between them is $r_1-r_2=128$, which is large. Thus the estimated chirp rates $\wh \gl_1(b)$ and $\wh \gl_2(b)$ should be easily obtained. In addition, if they are close to $r_1$ and $r_2$ respectively, then we will have accurate estimates for  $\phi^\gp_1(b)$  and $\phi^\gp_2(b)$. Here we use the same parameters as those used in Fig.\ref{fig:two_chirp_signal}, especially, $\sigma$ is constant, namely $\gs=0.023$.

As we see from our theorems that for a given multicomponent signal,  the key for the success of our method to recover its modes is: (i) For each $b$, can we obtain $\wh a_\ell(b)$ and $\wh \gl_\ell(b)$? and (ii) if yes for Question (i), then whether $\wh a_\ell(b)$ and $\wh \gl_\ell(b)$  are close to $\mu/\phi^\gp_\ell(b)$ and $\phi^{\gp\gp}_\ell(b)$? The answer to Question (ii) is guaranteed by the error bounds in our theorems. So the most important step is whether 
we can obtain $\wh a_\ell(b)$ and $\wh \gl_\ell(b)$. For this example of the two-component signal, the question is 
for each $b$,  
two peaks of  the function $h(a, \gl):=|U_x(a,b,\gl)|$ with $a\in (0, \infty), \gl\in \R$ 
are far apart enough from each other so that we can easily obtain the (local) maximum points $(\wh a_1, \wh \gl_1)$ and $(\wh a_2, \wh \gl_2)$ in the scale-(chirp rate) plane?  As examples, in the middle-left panel of Fig.\ref{fig:two_chirp_QP_CWT}, we show $h(a, \gl)$ with $b=32/256$; while  $h(a, \gl)$ with $b=160/256$ is presented in the  
middle-right panel of Fig.\ref{fig:two_chirp_QP_CWT}.  From these two panels, we observe that for either $b=32/256$ or $b=160/256$, two peaks of  $h(a, \gl)$ do be far apart and hence we should easily obtain  $(\wh a_1, \wh \gl_1)$ and $(\wh a_2, \wh \gl_2)$. Also observe from these two panels, the scale coordinates $\wh a_1, \wh a_2$ of the $(\wh a_1, \wh \gl_1)$ and $(\wh a_2, \wh \gl_2)$ change for different $b=32/256$ or $b=160/256$; while chirp rate coordinates $\wh \gl_1, \wh \gl_2$ 
essentially stay the same  (around $70$ and $-60$ respectively). This is due to the fact that $\phi^\gp_1(b)$ and $\phi^\gp_2(b)$  change with time $b$, while $\phi^{\gp\gp}_1(b)$ and $\phi^{\gp\gp}_2(b)$ are independent of $b$. 
In the bottom row of Fig.\ref{fig:two_chirp_QP_CWT}, we provide slices $|U_x(a,b,\gl)|$ with $a=2^{\frac{51}{32} }/256$ and $a=1/64$. All these pictures demonstrate that the components $x_1$ and $x_2$ are well-separated in the 3-dimensional space of adaptive TSC-R, although their IF curves are crossover.

 %%%%%%%%%%%%%%%%%%%the beginning of figure 2 %%%%%%%%%%%%%%%
\begin{figure}[th]
	\centering
	\hspace{-0.7cm}
	\begin{tabular}{c@{\hskip -0.2cm}c@{\hskip -0.2cm}c}		
		\resizebox{2.2in}{1.65in}{\includegraphics{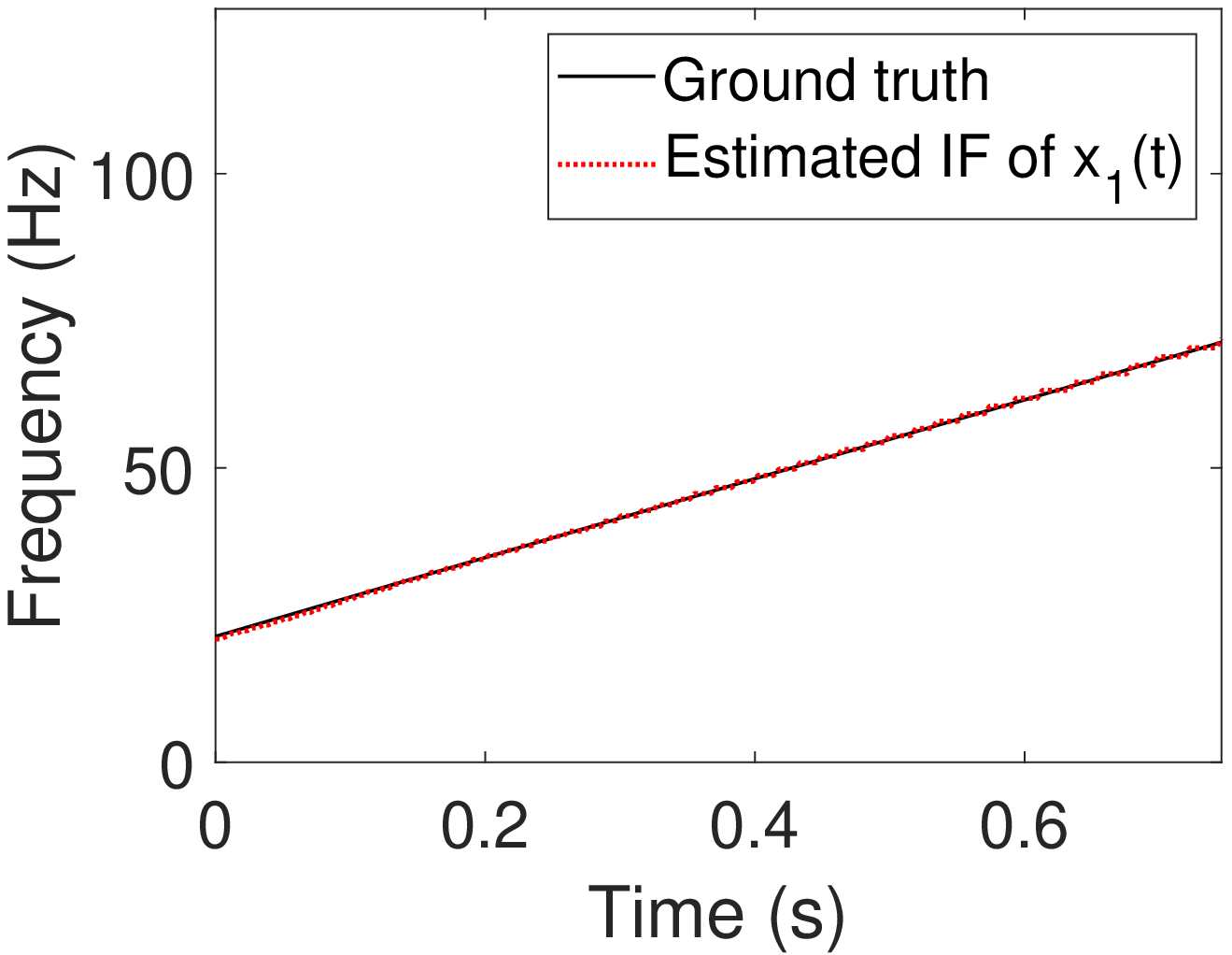}} \quad &
		\resizebox{2.2in}{1.65in}{\includegraphics{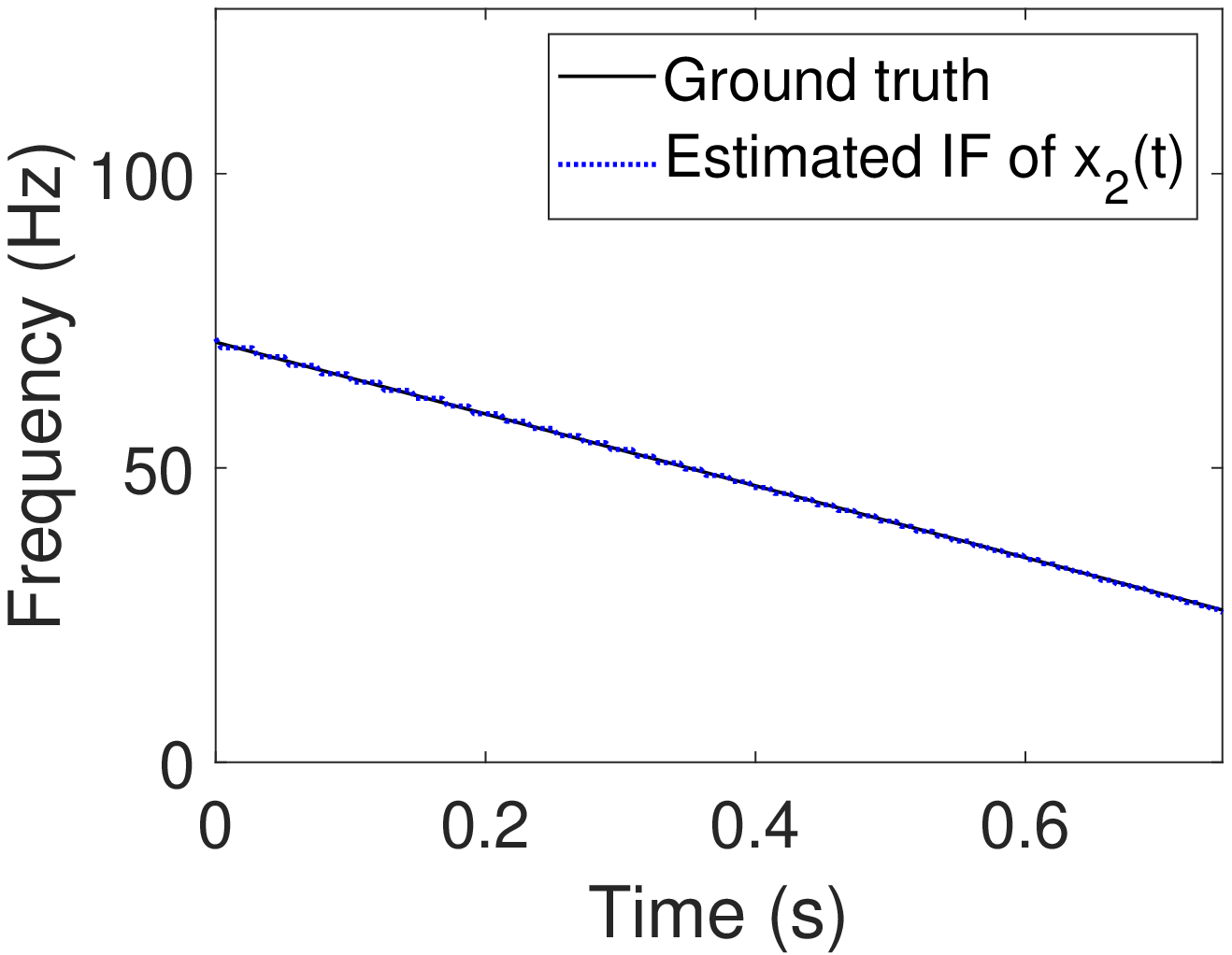}} \quad &
		\resizebox{2.2in}{1.65in}{\includegraphics{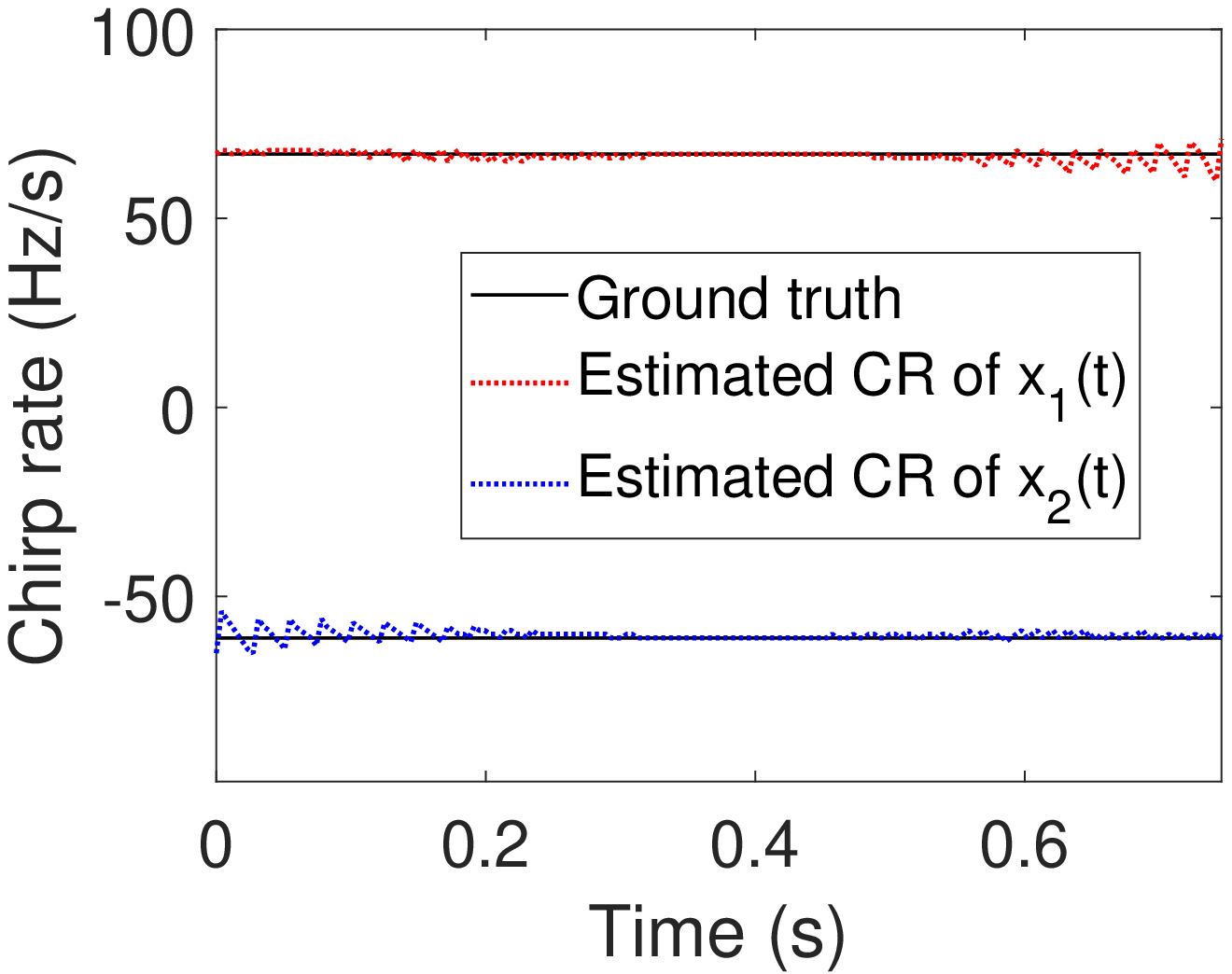}}		
	\end{tabular}
	\caption{\small %IFs and chirp rates estimate results by adaptive TSC-R.  
		Estimated IF of $x_1$ (left panel),	estimated IF of $x_2$ (middle panel) and estimated chirp rates (right panel) of two components by our method TSC-R.
	}
	\label{fig:two_chirp_QP_CWT_IF}
\end{figure}
%%%%%%%%%%%%%%%%the end of figure 2 %%%%%%%%%%%%%%%%%%%%%

 Fig.\ref{fig:two_chirp_QP_CWT_IF} shows the estimated IFs $\frac 1{\wh a_1(b)}, \frac1{\wh a_2(b)}$  and estimated chirp rates  $\wh \gl_1(b)$ and $\wh \gl_2(b)$.  Observe the estimated IFs are very close to ground truth $\phi^{\gp}_1(b)$ and $\phi^{\gp}_2(b)$.  The estimation errors of IFs and chirp rates are mainly caused by the bound effect, which can also be improved with a time-varying $\gs(b)$.    
\hfill $\blacksquare$

%%%%%%%%%%%%%%%%%%%the beginning of figure 4 %%%%%%%%%%%%%%%
\begin{figure}[th]
	\centering
	\hspace{-0.7cm}
	\begin{tabular}{c@{\hskip -0.2cm}c @{\hskip -0.2cm}c}
		\resizebox{2.2in}{1.65in}{\includegraphics{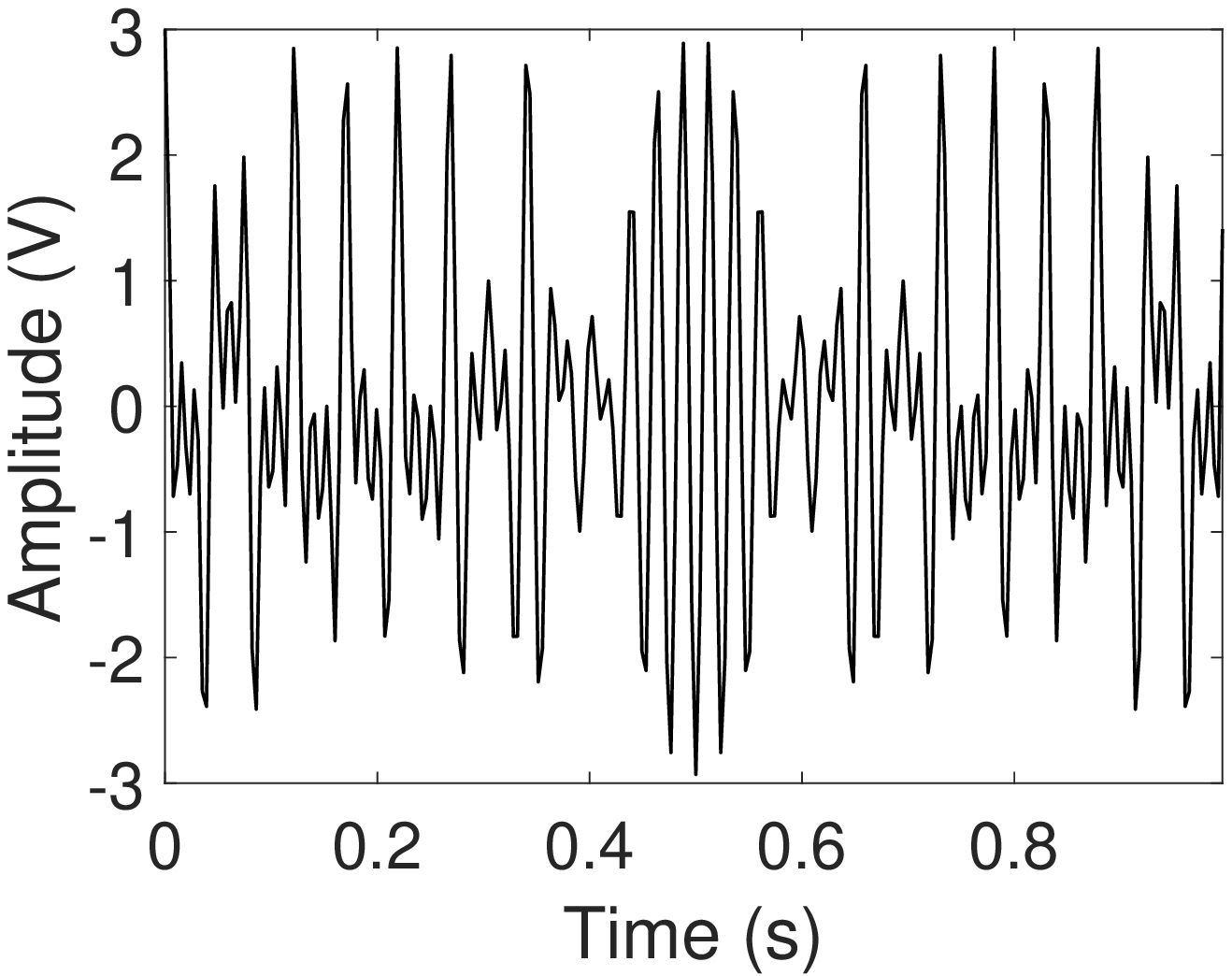}} \quad &
		\resizebox{2.2in}{1.65in}{\includegraphics{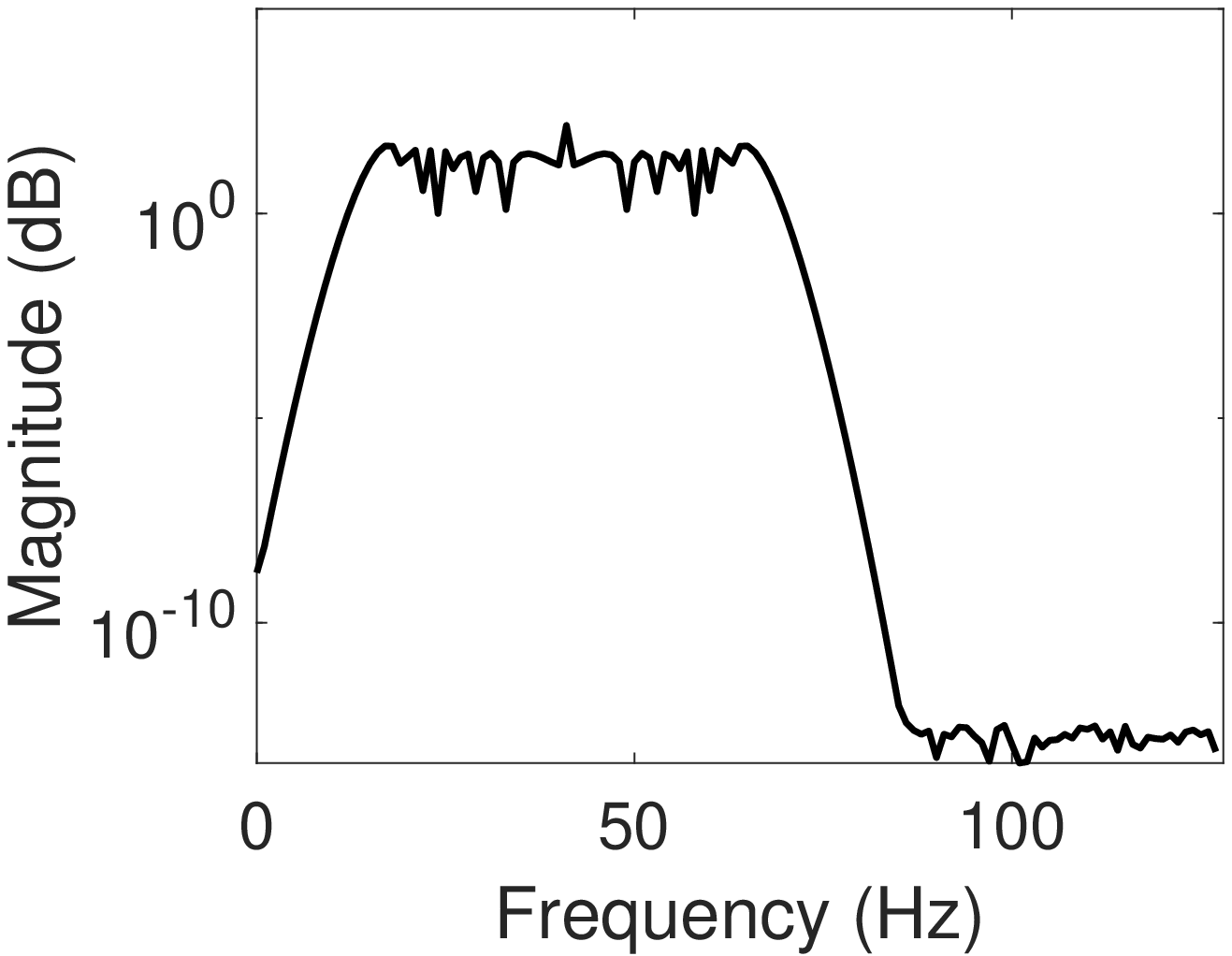}}\quad & 
		\resizebox{2.2in}{1.65in}{\includegraphics{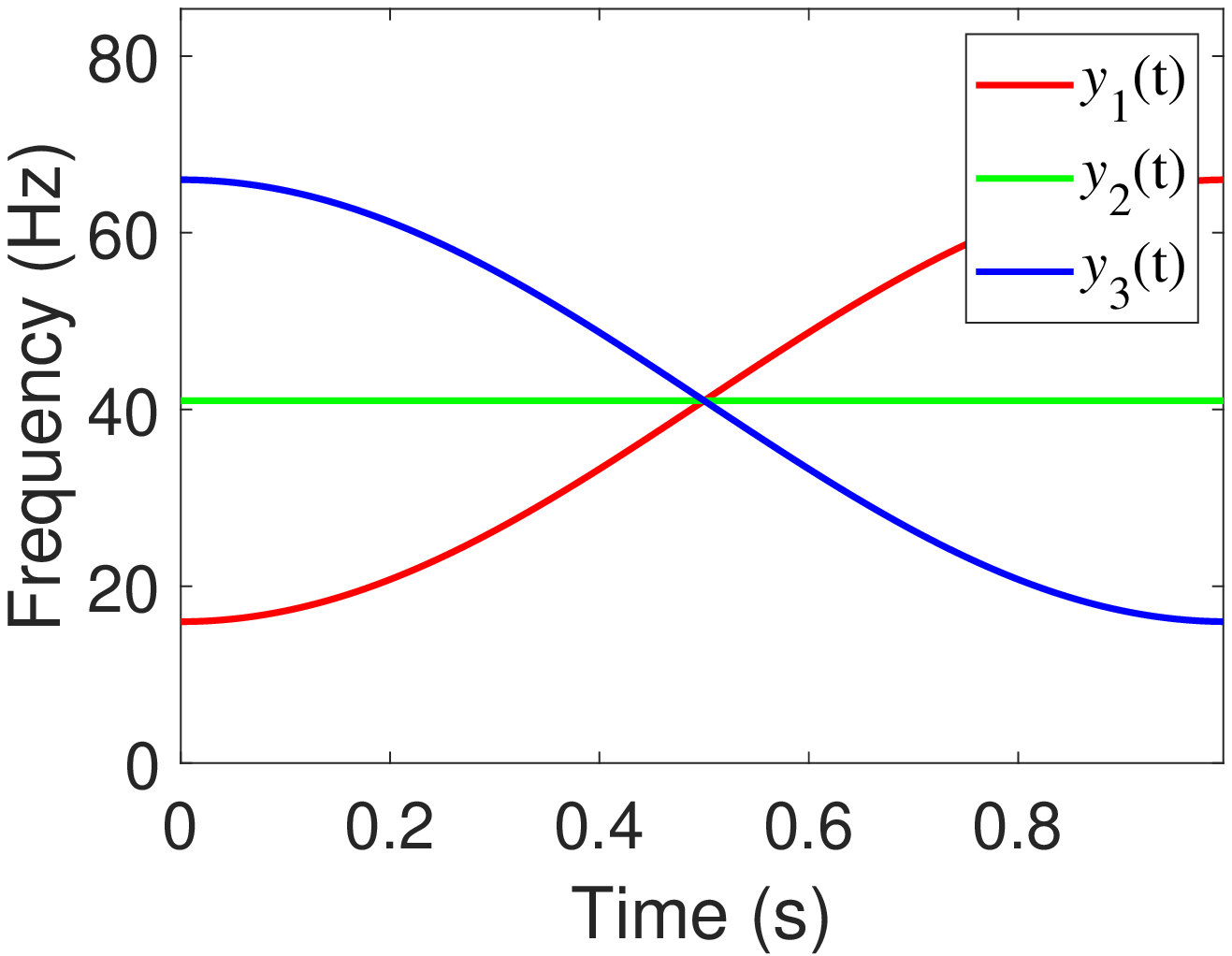}}\\ 
		\resizebox{2.2in}{1.65in}{\includegraphics{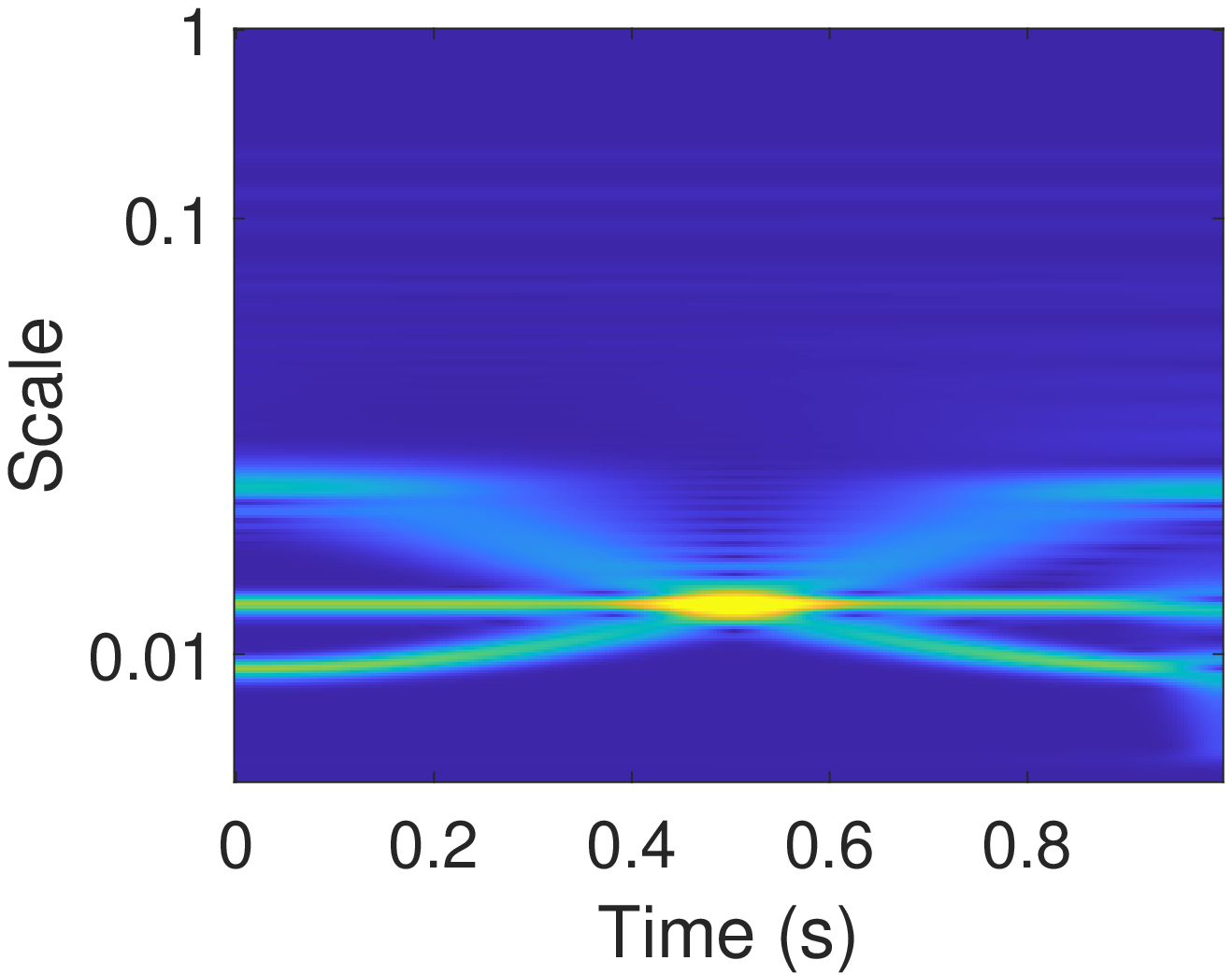}}\quad & 
		\resizebox{2.2in}{1.65in}{\includegraphics{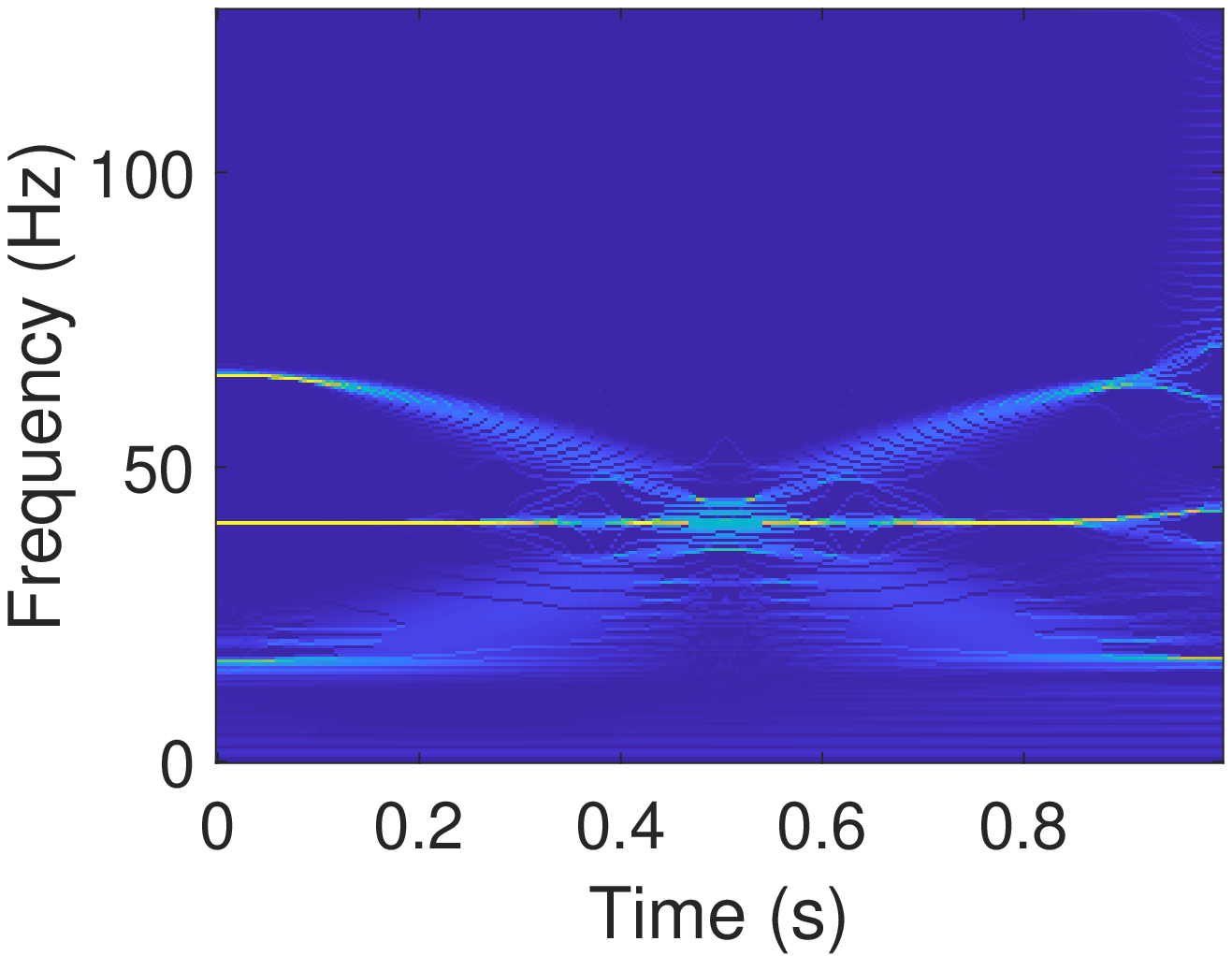}} \quad & 
		\resizebox{2.2in}{1.65in}{\includegraphics{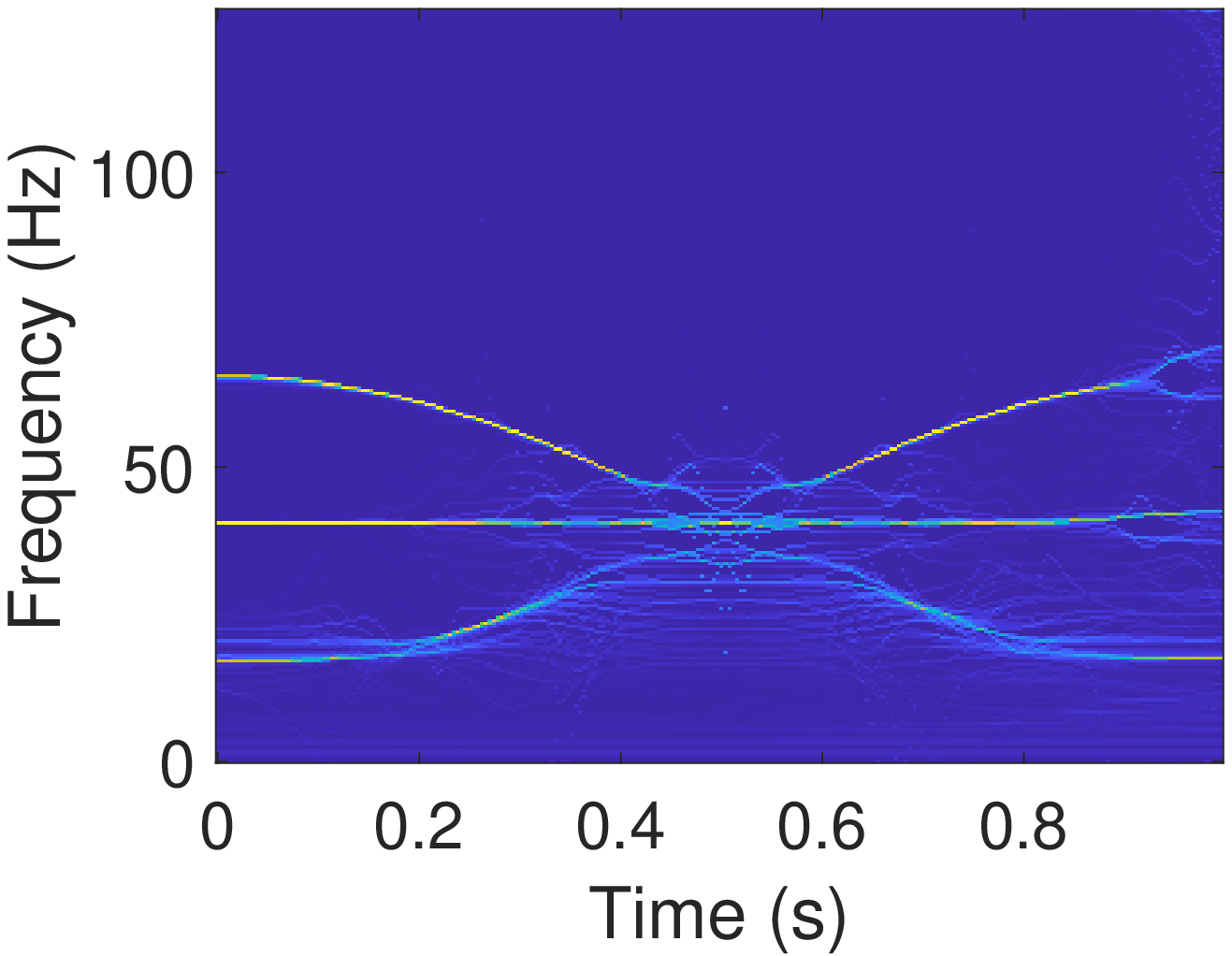}}
	\end{tabular}
	\caption{\small Three-component signal $y(t)$ given in \eqref{three_comp} and its time-frequency representations with SST. 
	 Top row (from left to right):  Waveform of $y(t)$, magnitude spectrum and ground truth IFs of $y_1(t)$, $y_2(t)$ and $y_3(t)$; Bottom row (from left to right): CWLT of $y(t)$,  CWT-based SST and CWT-based second-order SST.}
	\label{fig:three_comp_signal}
\end{figure}
%%%%%%%%%%%%%%%%the end of figure 4  %%%%%%%%%%%%%%%%%%%%%

\begin{example}
{\rm Let $y(t)$ be a truncation of the synthetic micro-Doppler signal in Fig.1, given as,
\begin{eqnarray} 
\nonumber
y(t)\hskip -0.6cm && =y_1(t)+y_2(t)+y_3(t)\\
\label{three_comp} &&=
\cos \left(82\pi t + 50 \cos \big (\pi t + \frac{\pi}{2} \big ) \right) + \cos (82 \pi t) 
+ \cos \left(82\pi t + 50 \cos \big (\pi t - \frac{\pi}{2} \big ) \right),
\end{eqnarray}
where $t\in[0,1)$. }
\end{example}

In the following experiment, $y(t)$ is discretized with the sampling rate {\rm 256Hz}, namely $t=0, \frac{1}{256}, \dots, \frac{255}{256}$. 
The IFs of $y_1$, $y_2$ and  $y_3$ are $\phi_1'(t) = 41- 25\sin(\pi t + \frac{\pi}{2})$, $\phi_2'(t) = 41$ and $\phi_3'(t) = 41- 25\sin(\pi t - \frac{\pi}{2})$, respectively. See the top-right panel in Fig.\ref{fig:three_comp_signal} for IFs.  The bottom row of Fig.\ref{fig:three_comp_signal} shows CWLT, SST and the second-order SST, where parameter $\sigma=0.023$ and the number of voice $n_v=32$.
Observed that CWLT and SST  are hardly to represent any of these three sub-signals separated and reliably. Thus they cannot separate sub-signals. Actually, as far as we known, there is no efficient algorithm available to recover the three components with the 256 points observation of $y(t)$ above. 
%%%%%%%%%%%%%%%%%%%the beginning of figure 5%%%%%%%%%%%%%%%
\begin{figure}[th]
	\centering
	\hspace{-0.7cm}
	\begin{tabular}{c@{\hskip -0.2cm}c}
		\resizebox{2.4in}{1.8in}{\includegraphics{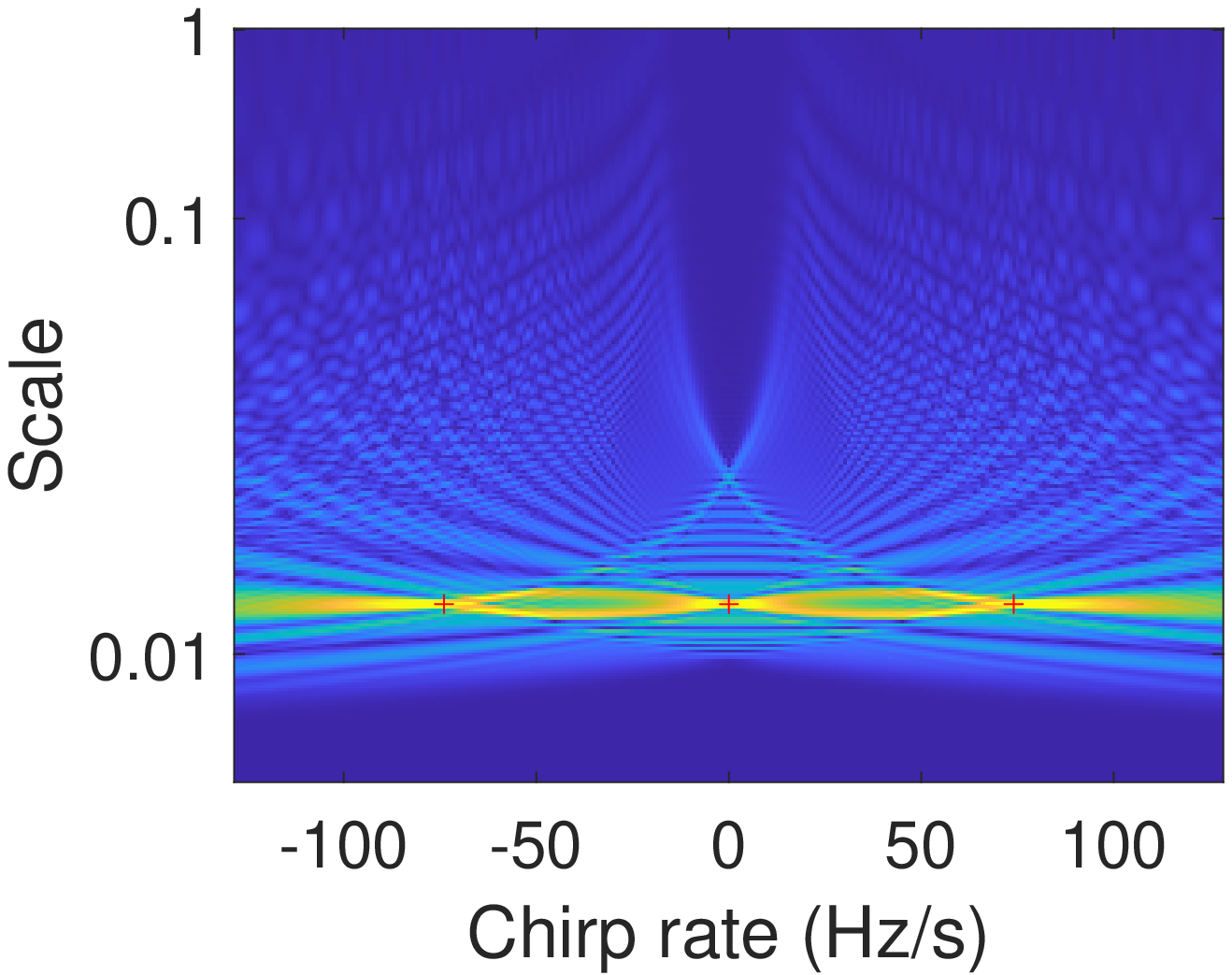}} \quad &
		\resizebox{2.4in}{1.8in}{\includegraphics{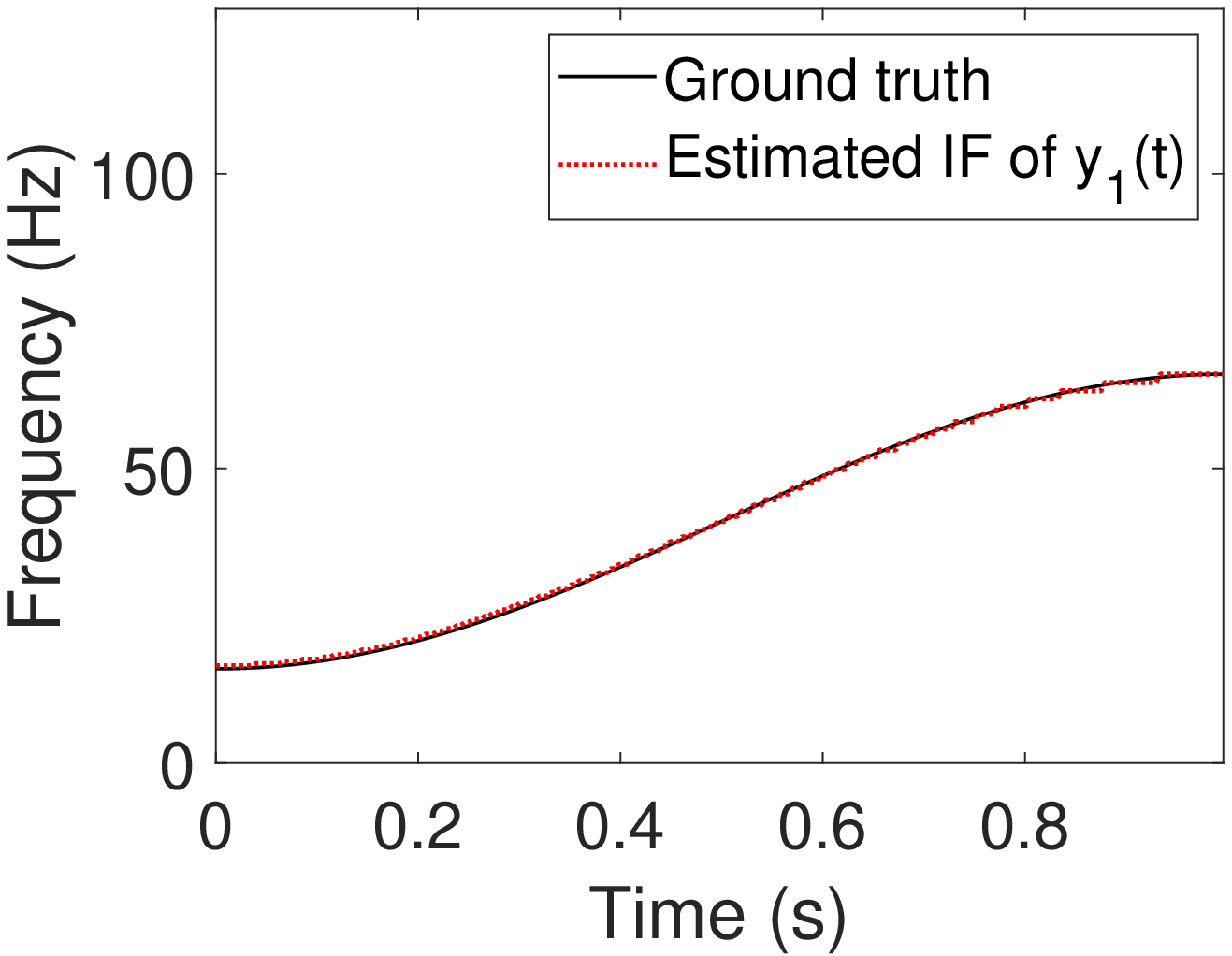}}\\ 
		\resizebox{2.4in}{1.8in}{\includegraphics{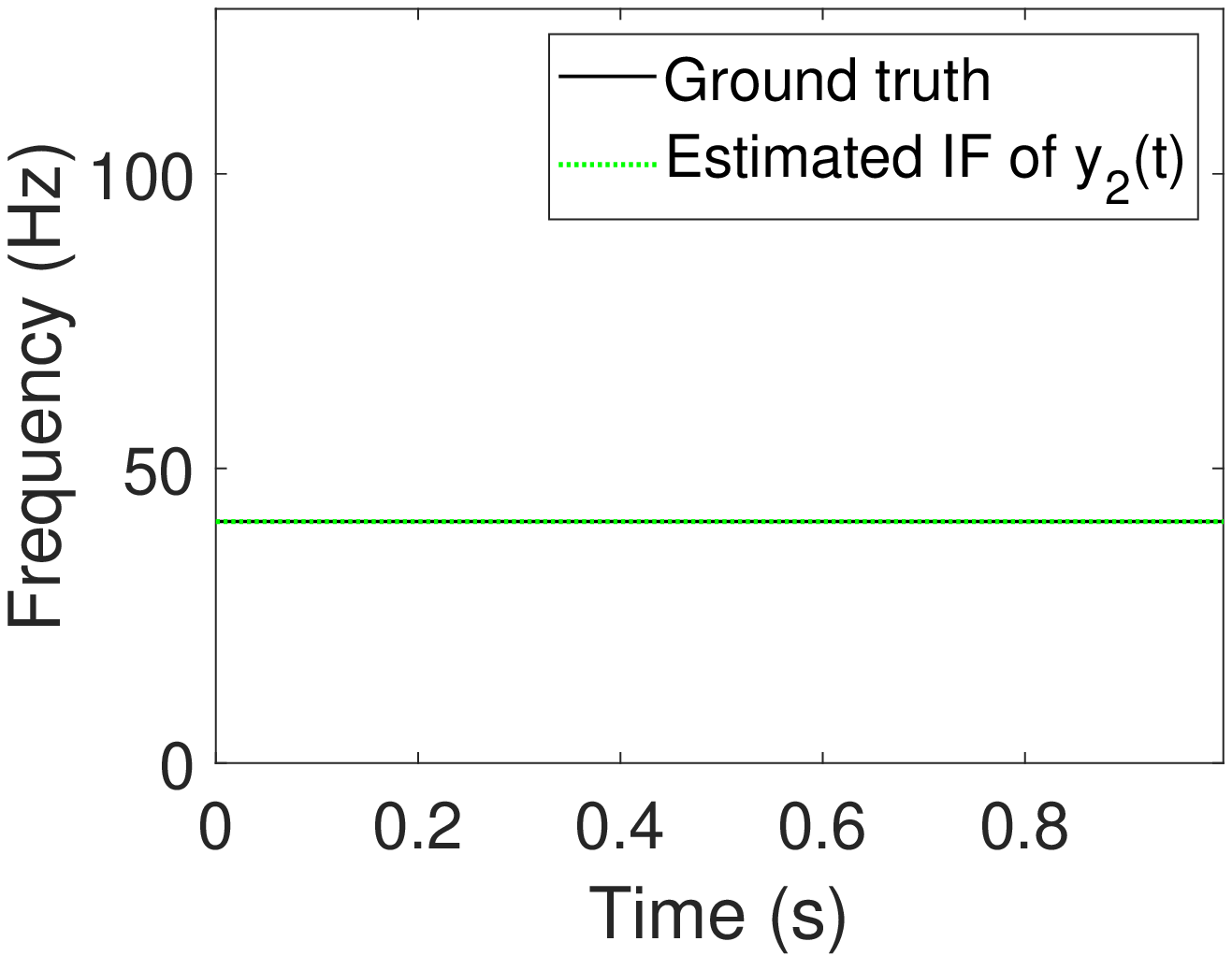}} \quad &
		\resizebox{2.4in}{1.8in}{\includegraphics{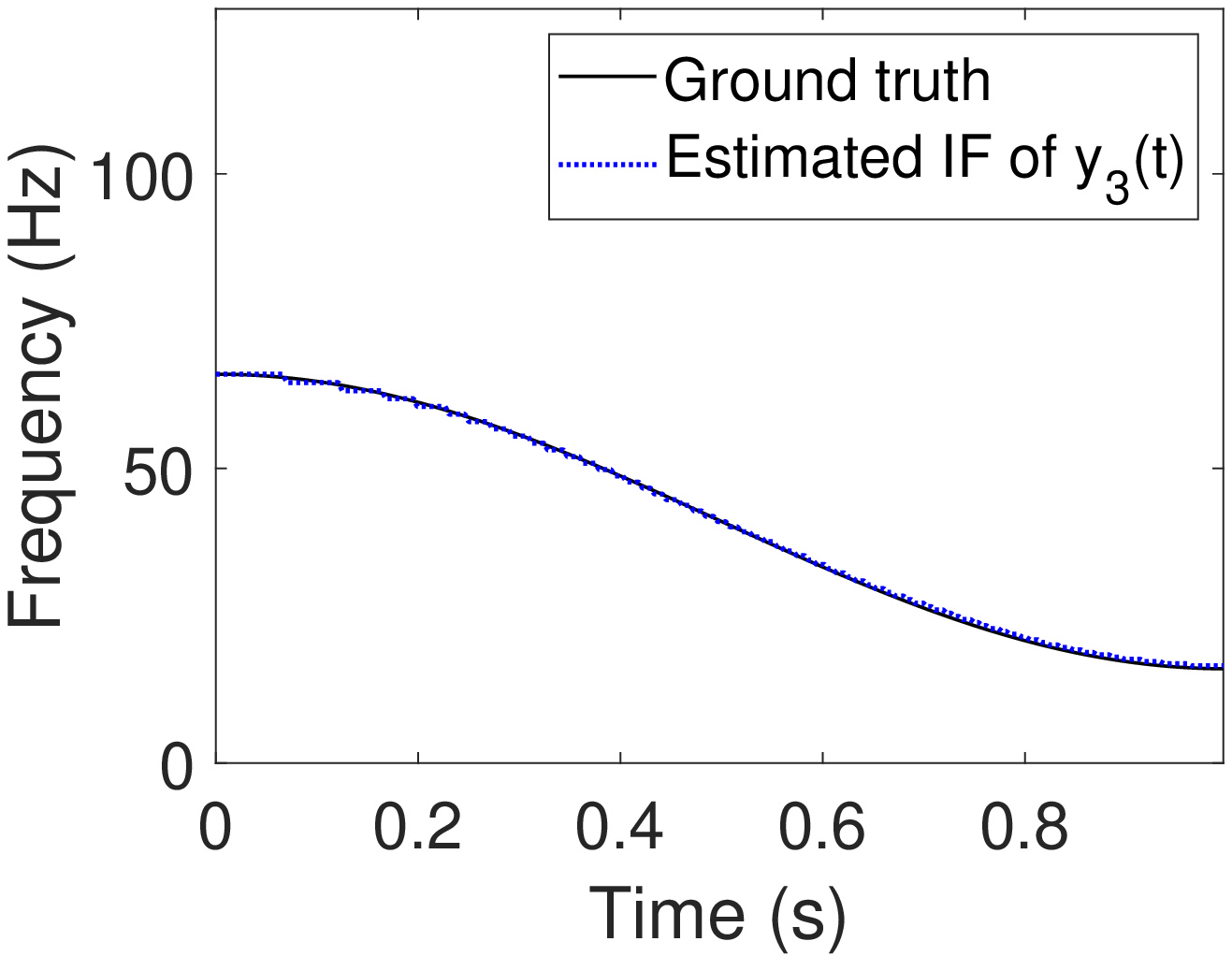}}\\ 
	\end{tabular}
	\caption{\small Slice of adaptive TSC-R $|U_y(a,b,\lambda)|$ of $y(t)$ when $b=0.5$ (Top-left panel) and estimated IF (dotted lines) of three components (Top-right and bottom panels).}
	\label{fig:three_comp_IFs_estimation}
\end{figure}
%%%%%%%%%%%%%%%%the end of figure 5%%%%%%%%%%%%%%%%%%%%%

The top-left panel in Fig.\ref{fig:three_comp_IFs_estimation} shows the slice of the adaptive TSC-R $|U_y(a,b,\lambda)|$ of $y(b)$ when $b=0.5$, namely the specific time point when the IFs of the three components are crossover. Observe 
even at this particular time $b=0.5$, the three peaks of $|U_y(a,0.5,\lambda)|$ (marked by $+$) are far apart in the scale-(chirp rate) plane. Thus we can easily obtain  $(\wh a_1, \wh \gl_1)$, $(\wh a_2, \wh \gl_2)$ and $(\wh a_3, \wh \gl_3)$ for this $b$. Actually 
for other $b$, three peaks of $|U_y(a, b,\lambda)|$ in  the scale-(chirp rate) plane are also far apart, and  much clearer and sharper than the case here when $b=0.5$. Hence,  these three components are well-separated in the 3-dimensional space of $(a, b, \gl)$.  
The estimated  IF of each component is given in the other panels of Fig.\ref{fig:three_comp_IFs_estimation}. The result shows our method is able to estimate the IF of each component correctly and precisely.  
We provide the result of component recovery in Fig.\ref{fig:three_comp_signal_recovery}, which shows that each recovered waveform  is very close to the corresponding mode, except for near time $b=0.5$, where the IFs of three components are crossover.     
The results show the validity and correctness of the proposed method.
Here we use a time-varying parameter $\gs(b)$, which improves IF estimation and mode recovery performance a lot. How to select $\gs(b)$ will be addressed in our future work. 
\hfill $\blacksquare$

%%%%%%%%%%%%%%%%%%%the beginning of figure 5%%%%%%%%%%%%%%%
\begin{figure}[th]
	\centering
	\hspace{-0.7cm}
	\begin{tabular}{c@{\hskip -0.2cm}c @{\hskip -0.2cm}c}
		\resizebox{2.2in}{1.65in}{\includegraphics{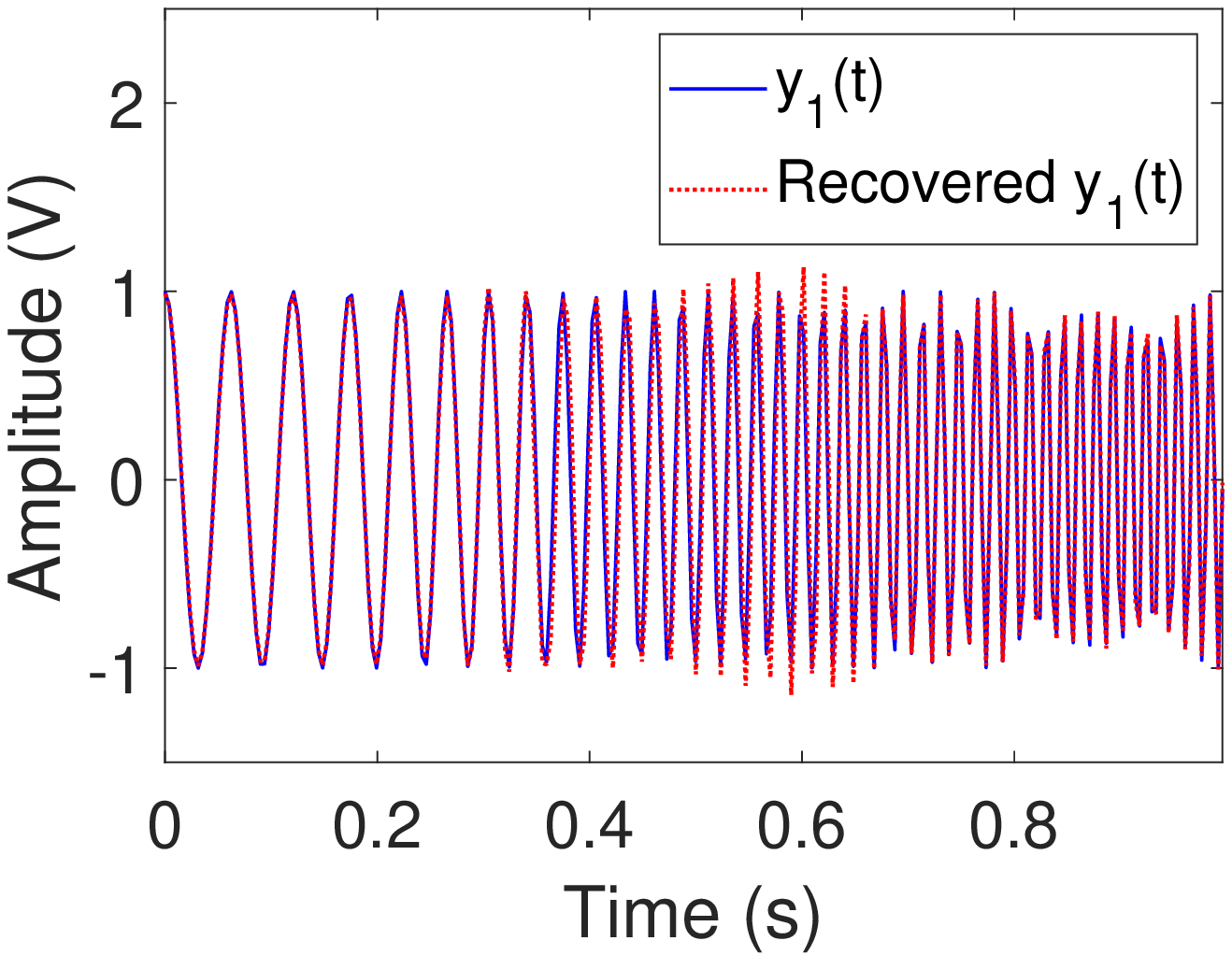}}\quad & 
		\resizebox{2.2in}{1.65in}{\includegraphics{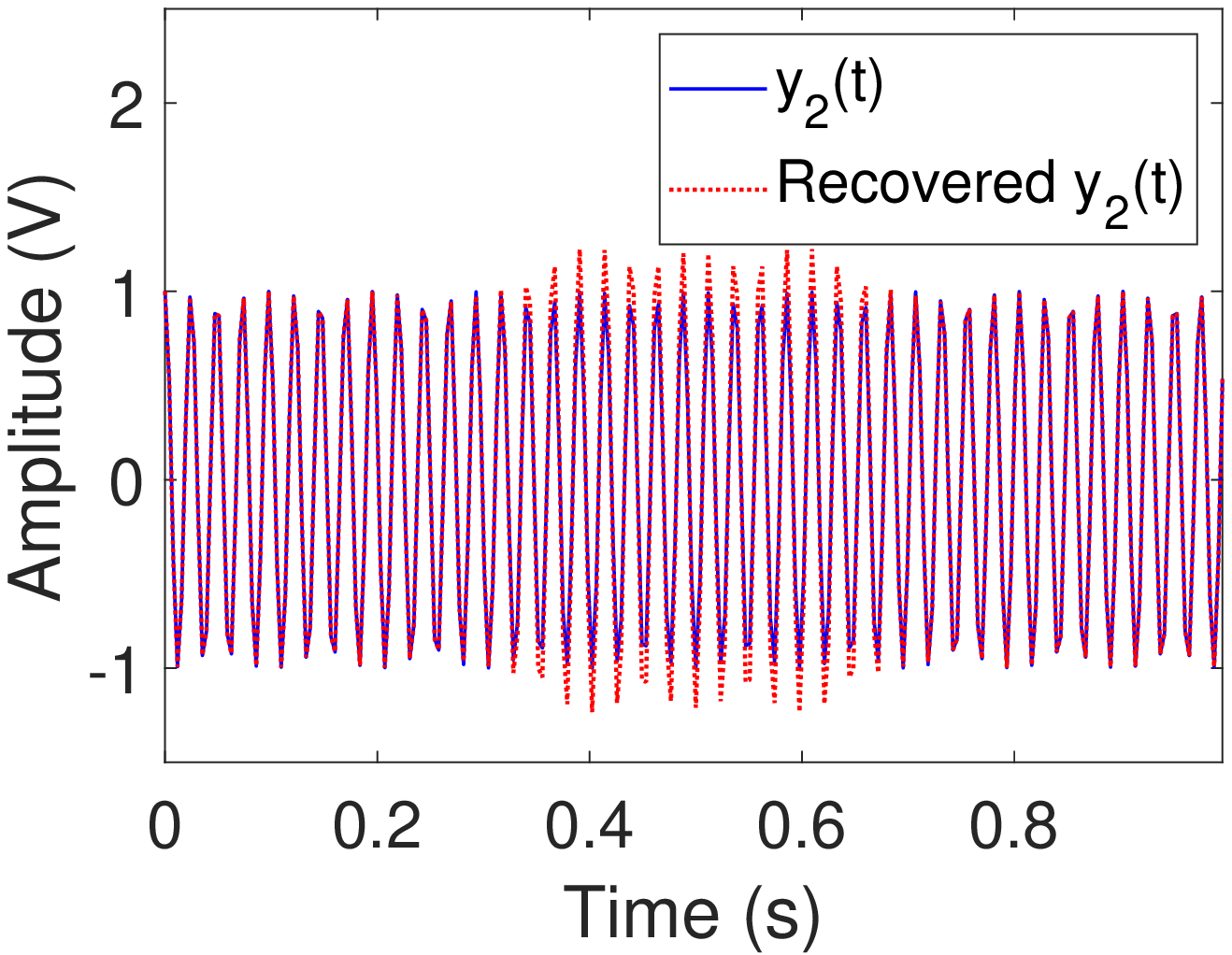}} \quad & 
		\resizebox{2.2in}{1.65in}{\includegraphics{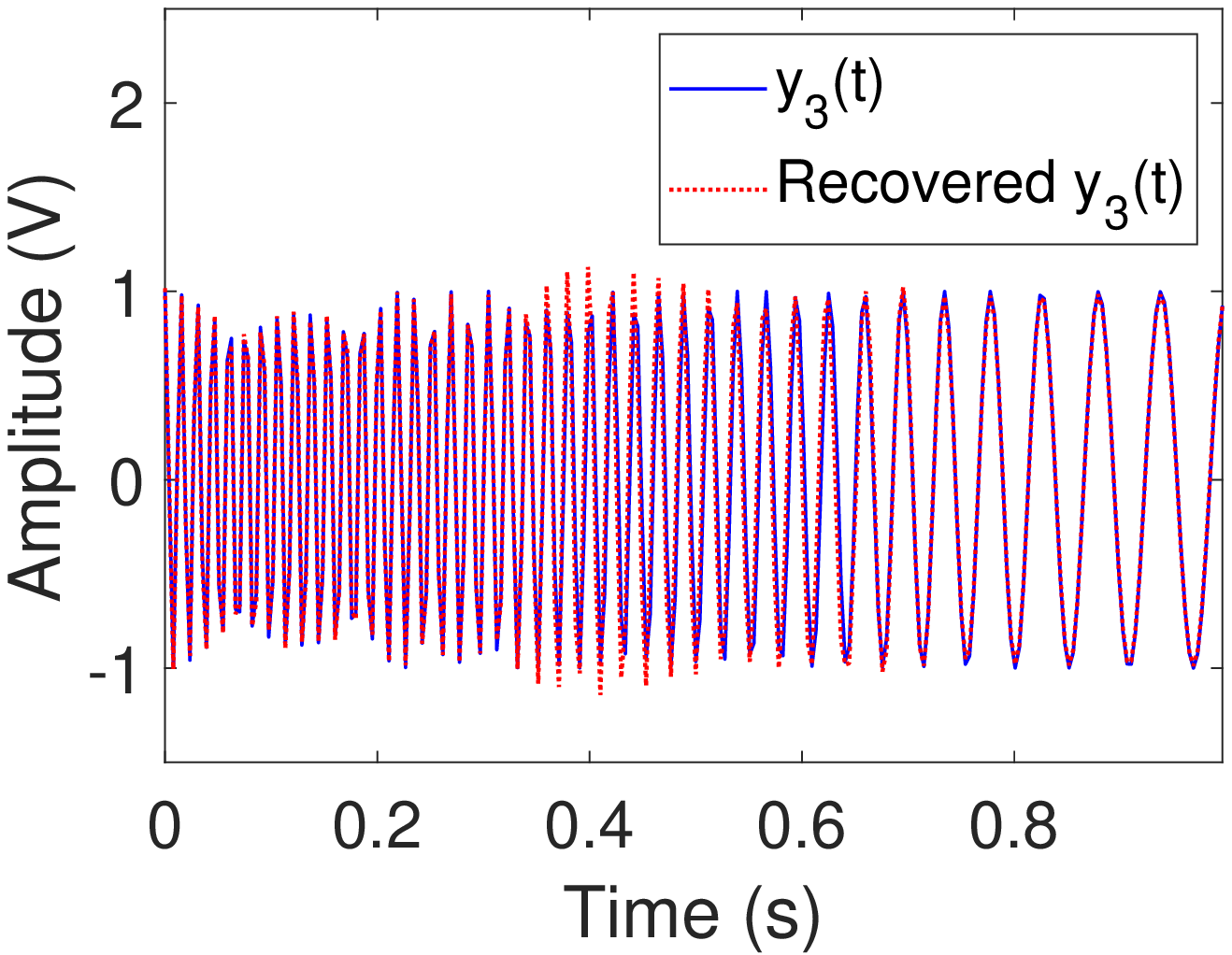}}
	\end{tabular}
	\caption{\small Recovered sub-signals (dotted red lines) of $y_1(t)$, $y_2(t)$ and $y_3(t)$ respectively (from left to right).}
	\label{fig:three_comp_signal_recovery}
\end{figure}
%%%%%%%%%%%%%%%%the end of figure 5%%%%%%%%%%%%%%%%%%%%%

\renewcommand{\arraystretch}{1.5} 
\begin{table}[tp]  
	
	\centering  
	\fontsize{10}{8}\selectfont  
	\begin{threeparttable}  
		\caption{IF estimate and component recovery errors under white Gaussian noise with different SNRs.}  
		\label{tab:error_SNR}  
		\begin{tabular}{|c|ccc|ccc|}  
			\toprule  
			\multirow{2}{*}{SNR}&  
			\multicolumn{3}{c|}{IF estimate errors}&\multicolumn{3}{c|}{Mode recovery errors}\cr  
			\cmidrule(lr){2-4} \cmidrule(lr){5-7}  
			&$\phi_1'(t)$&$\phi_2'(t)$&$\phi_3'(t)$&$y_1(t)$&$y_2(t)$&$y_3(t)$\cr  
			\midrule
			0 dB&0.1335&0.1361&0.0765&0.7830&0.6457&0.6037\cr    
			5 dB&0.0281&0.0092&0.0312&0.3183&0.2311&0.3390\cr  
			10 dB&0.0117&0.0077&0.0162&0.2056&0.1746&0.2506\cr  
			15 dB&0.0112&0.0005&0.0109&0.1926&0.1241&0.2137\cr  
			20 dB&0.0099&0.0005&0.0407&0.1927&0.1291&0.2053\cr  
			
			\bottomrule  
		\end{tabular}  
	\end{threeparttable}  
\end{table}

Finally, let us consider the effect of our computational scheme for signal data with additive noise, 
by adding a noise $n(t)$ process to signal $y(t)$ to have a synthetic signal $z(t)$ contaminated by noise $n(t)$:
$$z(t)=y(t)+n(t).$$
Here we let $n(t)$ be a zero-mean Gaussian noise. Table 1 shows the IF estimate and mode recovery errors  with different signal-to-noise ratios (SNRs), where the SNR is defined by 
$10\log_{10} \frac{||y||_2}{||n||_2} $. The errors in Table 1 are defined by
$$
E_f = \frac{||f-\wt f||_2}{||f||_2},
$$
where $\wt f$ is the estimation of $f$.
$E_f$ is also called the normalized mean square error. 
Observe that for SNR$ \ge 10$, IF estimation is stable and mode recovery errors are reasonable small.

%\section {Conclusion}
%In this paper, we introduce  time-scale-chirp\_rate  (TSC-R) component recovery operator  for multi-component signal separation with crossover instantaneous frequency  (IF) curves. TSC-R  represents a signal in a 3-dimensional space of time, scale and chirp rate. A non-stationary composite signal with cross-over IFs could be well-separated in this 3-dimensional space.  Based on the approximation of the blind-source signal by linear chirps at an arbitrary time-instant and adaptive TSC-R, we propose an innovative signal separation approach which is suitable for signals with crossing IFs.
%The error bounds for IF estimation and sub-signal recovery are established.  Numerical experiments demonstrate that the proposed method is more accurate and consistent in signal separation than SST and other approaches. The proposed method should have a significant potential for a variety of engineering applications, including radar targets recognition, communication channel detection, fault monitoring in mechanical systems, etc.

\end{document}